\documentclass[12pt,a4paper]{amsart}

\title{Espaces de Berkovich sur $\Z$ : morphismes étales}
\author{Dorian Berger}

\usepackage[french]{babel}
\usepackage[T1]{fontenc}
\usepackage[utf8]{inputenc}
\usepackage{amsmath}
\usepackage{xypic}
\usepackage{amsthm}
\usepackage{amssymb}
\usepackage{mathrsfs}
\usepackage{gensymb}
\usepackage{lmodern}
\usepackage{color}
\usepackage{tikz-cd}
\usepackage{hyperref}
\usepackage[style=alphabetic, backend=biber]{biblatex}
\hypersetup{bookmarks=false}

\newtheorem{df}{Définition}[section]
\newtheorem{ex}[df]{Exemple}
\newtheorem{prop}[df]{Proposition}
\newtheorem{rem}[df]{Remarque}
\newtheorem{thm}[df]{Théorème}
\newtheorem{lem}[df]{Lemme}
\newtheorem{cor}[df]{Corollaire}
\newtheorem{nota}[df]{Notation}

\newtheorem*{intro_df}{Définition}
\newtheorem*{intro_prop}{Proposition}
\newtheorem*{intro_thm}{Théorème}

\def \A{{\mathcal{A}}}
\def \Aff{{\mathbb{A}}}
\def \B{{\mathcal{B}}}
\def \C{{\mathbb{C}}}
\def \F{{\mathcal{F}}}
\def \H{{\mathcal{H}}}
\def \I{{\mathcal{I}}}

\def \K{{\mathcal{K}}}
\def \m{{\mathfrak{m}}}
\def \M{{\mathcal{M}}}
\def \N{{\mathbb{N}}}
\def \O{{\mathcal{O}}}
\def \R{{\mathbb{R}}}
\def \Xsch{{\mathscr{X}}}
\def \Ysch{{\mathscr{Y}}}
\def \Ssch{{\mathscr{S}}}
\def \V{{\mathcal{V}}}
\def \Z{{\mathbb{Z}}}

\def \Frac{{\mathrm{Frac}}}
\def \Spec{{\mathrm{Spec}}}
\def \an{{\mathrm{an}}}
\def \sch{{\mathrm{sch}}}
\def \Id{{\mathrm{Id}}}
\def \red{{\mathrm{red}}}
\def \Hom{{\mathrm{Hom}}}
\def \ev{{\mathrm{ev}}}

\usepackage{scalerel,stackengine}
\stackMath
\newcommand\reallywidehat[1]{%
\savestack{\tmpbox}{\stretchto{%
  \scaleto{%
    \scalerel*[\widthof{\ensuremath{#1}}]{\kern-.6pt\bigwedge\kern-.6pt}%
    {\rule[-\textheight/2]{1ex}{\textheight}}
  }{\textheight}%
}{0.5ex}}%
\stackon[1pt]{#1}{\tmpbox}%
}

\newcommand\quotient[2]{\text{\raise1ex\hbox{$#1$}\Big/\lower1ex\hbox{$#2$}}}
    
\newcommand{\centerrightarrow}[2]{\begin{center} \begin{tikzcd}[ampersand replacement=\&] #1 \ar[r] \& #2 \end{tikzcd} \end{center}}

\newcommand{\Supp}[1]{\mathrm{Supp} \left( #1 \right)}

\makeatletter
\@namedef{subjclassname@2020}{%
  \textup{2020} Mathematics Subject Classification}
\makeatother



\begin{document}

\maketitle

\begin{abstract}
Nous étudions les morphismes non ramifiés, étales et lisses entre espaces de Berkovich sur $\Z$. Nous obtenons des analogues des propriétés classiques des morphismes de schémas ainsi que des critères par analyfication. Ces résultats sont aussi valables sur les corps valués complets, les anneaux d'entiers de corps de nombres et les anneaux de valuation discrète. Ces différents cas sont traités de façon unifiée.
\end{abstract}

\renewcommand{\abstractname}{Abstract (Berkovich spaces over Z : étale morphisms)}

\begin{abstract}
We develop properties of unramified, étale and smooth morphisms between Berkovich spaces over $\Z$. We prove that they satisfy properties analogous to those of morphisms of schemes and we provide analytification criteria. Our results hold for any valued field, rings of integers of a number field and discrete valuation rings. Those cases are treated by a unified way.
\end{abstract}

\tableofcontents

\newpage

\section*{Introduction}

Depuis le début du \textsc{XX}\ieme ~siècle, la théorie des nombres s'est enrichie de nombreux apports dus à l'analyse $p$-adique. Ces résultats ont motivé la construction d'une géométrie analytique sur les corps ultramétriques et, à partir des années 60, plusieurs théories sur ce sujet ont émergé. Parmi celles-ci, on peut noter la géométrie rigide de J. Tate (\cite{tate_rigid_1971}), les schémas formels de M. Raynaud (\cite{raynaud_geometrie_1974}), les espaces adiques de R. Huber (\cite{huber_continuous_1993} et \cite{huber_generalization_1994}) ou encore les espaces analytiques de V. Berkovich (\cite{berkovich_spectral_1990}). Cette dernière, à laquelle on doit de nombreuses applications, notamment dans le programme de Langlands, en théorie de Hodge $p$-adique et en dynamique, est riche de bonnes propriétés topologiques : les espaces y sont localement compacts, localement connexes par arc et localement contractiles. Un autre intérêt de la théorie de Berkovich est qu'elle permet la construction d'espaces analytiques sur un anneau de Banach quelconque. En particulier, on peut construire des espaces analytiques sur les anneaux $\C$ et $\Z$, tous deux munis de la valeur absolue usuelle. Dans le premier cas, on retrouve exactement les espaces analytiques complexes tandis que, dans le second cas, on obtient des espaces fibrés en espaces analytiques complexes et $p$-adiques. Bien que ces exemples soient donnés dans le premier chapitre de \cite{berkovich_spectral_1990}, les espaces analytiques sur $\Z$ ne sont pas plus étudiés dans cet ouvrage. La première description approfondie d'un tel espace est celle, due à J. Poineau, présentée dans \cite{poineau_droite_2010} et qui traite le cas de la droite affine analytique.
\medbreak
Dans cet article, on se propose d'étudier les morphismes étales entre espaces analytiques sur $\Z$ ou sur un autre anneau d'entiers de corps de nombres. Un tel morphisme induit un isomorphisme local entre les fibres complexes et un morphisme étale au sens de \cite{berkovich_etale_1993} entre les fibres $p$-adiques. On traite ces deux cas de façon unifiée. Les méthodes utilisées permettent d'obtenir les résultats sur une classe d'anneaux plus générale, comprenant les corps valués complets et les anneaux de valuation discrète.
\medbreak
Dans la première section, on rappelle la construction des espaces analytiques sur un anneau de Banach quelconque. Cette définition suit celle due à V. Berkovich et issue de \cite{berkovich_spectral_1990}. Ensuite, on définit la notion de morphisme d'espaces analytiques et on présente quelques résultats provenant de \cite{poineau_espaces_2013} et \cite{lemanissier_espaces_2020}. Dans la deuxième section, on donne la définition d'une classe d'anneaux de Banach contenant $\Z$ (définition \ref{def anneaux de base géométriques}) à laquelle on se restreindra tout au long de cet article. Pour le reste de cette introduction, $\A$ désignera un tel anneau. Des propriétés fondamentales principalement issues de \cite{lemanissier_espaces_2020} sont aussi présentées dans la section 2. L'étude à laquelle sont consacrées les sections 3 et 4 permet d'obtenir d'importants résultats sur les fibres des espaces analytiques. La principale difficulté provient du fait que la restriction aux fibres d'un morphisme d'espaces analytiques est plus subtile que son analogue schématique : si $f : X \rightarrow S$ est un tel morphisme et $s \in S$, la fibre $X_s = f^{-1}(s)$ est naturellement munie d'une structure d'espace analytique, non pas sur le corps résiduel de $\O_s$, qu'on notera $\kappa(s)$, mais sur son complété $\H(s) = \widehat{\kappa(s)}$. Parmi les résultats obtenus, on notera notamment les suivants, où $\m_x$ (resp. $\m_{X_s,x}$) désigne l'idéal maximal de l'anneau local $\O_x$ (resp. $\O_{X_s,x}$) :

\begin{intro_prop}[{Proposition \ref{rig épais ideal max fibre prop}}]
Soient $f : X \rightarrow S$ un morphisme d'espaces $\A$-analytiques, $x \in X$ et $s = f(x)$ tels que $\kappa(x)$ soit une extension finie de $\kappa(s)$. On a alors :
\begin{center}
$\m_{X_s,x} = \sqrt{\m_{x}\O_{X_s,x}}$.
\end{center}
\end{intro_prop}

\begin{intro_thm}[{Théorème \ref{plat fibre algébrique thm}}]
Soient $f : X \rightarrow S$ un morphisme d'espaces $\A$-analytiques, $x \in X$ et $s = f(x)$. Alors le morphisme \centerrightarrow{\quotient{\O_x}{\m_s\O_x}}{\O_{X_s,x}} est plat.
\end{intro_thm}

Ces résultats sont utiles pour démontrer certains critères par fibres pour les propriétés de morphismes étudiées dans cet article et dont voici la définition :

\begin{intro_df}
Soient $f : X \rightarrow S$ un morphisme d'espaces $\A$-analytiques, $x \in X$ et $s = f(x)$. On dit que $f$ est \emph{plat} (resp. \emph{non ramifié}, resp. \emph{étale}) en $x$ si le morphisme induit $f^\sharp_x : \O_s \rightarrow \O_x$ est plat (resp. non ramifié, resp. étale).
\end{intro_df}

Parmi ces critères, on notera le suivant :

\begin{intro_prop}[{Corollaire \ref{plat par fibres cor}, Proposition \ref{étale par fibres prop}}]
Soient $S$ un espace $\A$-analytique, $f : X \rightarrow Y$ un morphisme d'espaces $S$-analytiques, $x \in X$, $y = f(x)$ et $s \in S$ l'image de $x$. On suppose que $X \rightarrow S$ est plat en $x$. Si $f_s : X_s \rightarrow Y_s$ est plat (resp. étale) en $x$ alors $f : X \rightarrow Y$ est plat (resp. étale) en $x$.
\end{intro_prop}

Plus précisément, la section 3 est consacrée à l'étude d'une propriété ponctuelle de morphisme satisfaite par les morphismes non ramifiés :

\begin{intro_df}
Soient $f : X \rightarrow S$ un morphisme d'espaces $\A$-analytiques, $x \in X$ et $s = f(x)$. On dit que $f$ est \emph{rigide épais} en $x$ si $\kappa(x)$ est une extension finie de $\kappa(s)$.
\end{intro_df}

Cette notion est fondamentale dans la démonstration des différents critères par fibres, avec notamment le fait suivant : si $f$ est rigide épais en $x$, alors $\O_x/\m_{f(x)}\O_x$ est un corps si et seulement si $\O_{X_{f(x)},x}$ en est un (proposition \ref{rig épais corps sur les fibres prop}). Dans la section 4, on applique certains résultats de la section 3 à l'étude des morphismes plats. On en déduit notamment le critère de platitude par fibres énoncé ci-dessus ainsi que la platitude des morphismes de projection $\Aff^n_S \rightarrow S$ (corollaire \ref{plat projection affine cor}) et des morphismes lisses (proposition \ref{lisse plat prop}). Dans la section 5, on établit des résultats sur la dimension de Krull des anneaux locaux de certains points particuliers dont on se sert ensuite pour montrer que les morphismes non ramifiés sont quasi-finis (corollaire \ref{non ramifié isolé dans fibre cor}).
\medbreak
Les résultats sur les morphismes non ramifiés commencent à proprement parler dans la section 6. On y propose une étude systématique de ces morphismes apparus comme outils dans \cite{lemanissier_topology_2019}. Leur structure locale est donnée dans la section 7. On y démontre notamment le critère suivant : un morphisme est non ramifié en un point $x$ si et seulement si sa diagonale est un isomorphisme local en $x$ (proposition \ref{non ramifié diagonale prop}). Les morphismes étales sont étudiés dans la section 8. Outre le critère par fibres déjà énoncé, on y démontre dans le cadre des espaces analytiques certaines propriétés classiques des morphismes de schémas étales comme le caractère ouvert (corollaire \ref{étale ouvert cor}) ou l'invariance topologique (proposition \ref{étale invariance topologique prop}). Dans la neuvième et dernière section, on introduit les morphismes lisses et on montre quelques résultats dont la caractérisation suivante : un morphisme d'espaces analytiques est étale en un point si et seulement s'il est lisse et non ramifié en ce point (corollaire \ref{lisse non ramifié implique étale cor}).
\medbreak
On dispose d'un foncteur d'analytification des $\A$-schémas localement de présentation finie $\Xsch \mapsto \Xsch^\an$ et on note $\rho : \Xsch^\an \rightarrow \Xsch$ le morphisme canonique associé. On montre que, quitte à ajouter certaines hypothèses sur l'anneau de base (voir définition \ref{def anneau de Dedekind analytique}), ce foncteur conserve différentes propriétés de morphisme étudiées dans cet article, répondant ainsi à une conjecture énoncée dans \cite[(2.2.9)]{lemanissier_topology_2019} :

\begin{intro_thm}[{Corollaire \ref{non ramifié analytification cor}, Proposition \ref{étale analytification prop}}]
Soient $f : \Xsch \rightarrow \Ssch$ un morphisme entre $\A$-schémas localement de présentation finie et $x \in \Xsch^\an$. Alors $f$ est non ramifié (resp. étale) en $\rho(x)$ si et seulement si $f^{\an} : \Xsch^{\an} \rightarrow \Ssch^{\an}$ est non ramifié (resp. étale) en $x$.
\end{intro_thm}

Il y a cependant quelques propriétés qui diffèrent entre les morphismes étales de schémas et les morphismes étales d'espaces analytiques. L'exemple le plus flagrant est certainement la présentation locale de tels morphismes :

\begin{intro_prop}[{Corollaire \ref{étale Chevalley cor}}]
Soient $f : X \rightarrow S$ un morphisme d'espaces $\A$-analytiques, $x \in X$ et $s = f(x)$. Alors $f$ est étale en $x$ si et seulement si on dispose d'un polynôme unitaire $P(T) \in \O_s[T]$ dont l'image dans $\kappa(s)[T]$ est irréductible et séparable et tel que $f$ induise un isomorphisme : \[ \quotient{\O_s[T]}{(P(T))} \cong \O_x. \]
\end{intro_prop}

L'analogue schématique de ce résultat nécessite une opération de localisation supplémentaire dont on peut se passer ici grâce aux bonnes propriétés topologiques des espaces en jeu. Ce phénomène est illustré dans la remarque \ref{étale exemple présentation locale rem}.

\subsection*{Remerciements}

Nous remercions Jérôme Poineau pour les nombreuses discussions à l'origine de cet article, ainsi que Thibaud Lemanissier pour ses remarques.

\section{Espaces $\A$-analytiques}

L'objectif de cette section est de rappeler la définition d'espace analytique sur un anneau de Banach quelconque présentée dans le premier chapitre de \cite{berkovich_spectral_1990}. Le lecteur trouvera une référence plus précise en \cite{lemanissier_espaces_2020}, notamment dans les chapitres 1 et 2.
\smallbreak
Dans cette section, $n \in \N$ et $(\A,\|.\|_\A)$ désigne un anneau de Banach, c'est-à-dire un anneau normé et complet pour cette norme.

\begin{df}
Une application $|.| : \A[T_1, \dots, T_n] \rightarrow \R_+$ est une \emph{semi-norme multiplicative sur $\A[T_1, \dots, T_n]$ bornée sur $\A$} si, pour tout $P,Q \in \A[T_1, \dots, T_n]$ et pour tout $a \in \A$, on a :
\begin{itemize}
\item $|0| = 0$,
\item $|1| = 1$,
\item $|P+Q| \leq |P|+|Q|$,
\item $|PQ| = |P||Q|$,
\item $|a| \leq \|a\|_\A$.
\end{itemize} 
\end{df}

\begin{df}
L'ensemble des semi-normes multiplicatives sur $\A[T_1, \dots, T_n]$ bornées sur $\A$ est appelé \emph{espace affine analytique de dimension $n$ sur $\A$} et est noté $\Aff^n_\A$. Dans le cas $n = 0$, cet ensemble est noté $\M(\A)$.
\smallbreak
On munit $\Aff^n_\A$ de la topologie de la convergence simple, c'est-à-dire la topologie la plus grossière telle que les applications d'évaluations $\Aff^n_\A \rightarrow \R_+$, $|.| \mapsto |P|$ soient continues pour $P \in \A[T_1, \dots, T_n]$.
\end{df}

\begin{prop}[{\cite[Corollaire~6.8]{poineau_espaces_2013}}]\label{projection affine ouverte prop}
Le morphisme naturel $\Aff^n_\A \rightarrow \M(\A)$ est ouvert.
\end{prop}

On pensera aux éléments de $\Aff^n_\A$ comme aux points d'un espace et, si $x \in \Aff^n_\A$, on notera $|.|_x : \A[T_1, \dots, T_n] \rightarrow \R_+$ la semi-norme associée. On cherche à présent à construire un faisceau structural sur $\Aff^n_\A$.

\begin{df}
Soit $x \in \Aff^n_\A$. L'idéal $\ker(|.|_x) \subset \A[T_1, \dots, T_n]$ est premier et $|.|_x$ induit une valeur absolue sur \[\Frac\left(\quotient{\A[T_1, \dots, T_n]}{\ker(|.|_x)}\right).\] Le complété de ce corps est appelé \emph{corps résiduel complété en $x$} et noté $\H(x)$.
\smallbreak
On note $\ev_x : \A[T_1, \dots, T_n] \rightarrow \H(x)$ l'application $f \mapsto |f|_x$. Si $f~\in~\A[T_1, \dots, T_n]$, $\ev_x(f)$ est noté $f(x)$.
\end{df}

\begin{df}
Soit $V \subset \Aff^n_\A$ une partie compacte. Le localisé de $\A[T_1, \dots, T_n]$ en $\{ f \in \A[T_1, \dots, T_n] \mid \forall x \in V, f(x) \neq 0 \}$ est appelé \emph{anneau des fractions rationnelles sans pôles sur $V$} et noté $\K(V)$. L'application $\A[T_1, \dots, T_n] \rightarrow \R_+$, $P \mapsto \max_{x \in V}(|P|_x)$ s'étend en une semi-norme sur $\K(V)$, appelée \emph{semi-norme uniforme sur $V$} et notée $\|.\|_V$.
\end{df}

\begin{df}
Soit $U \subset \Aff^n_\A$ un ouvert. L'ensemble des applications $f : U \rightarrow \bigsqcup_{x \in U} \H(x)$ vérifiant, pour tout $x \in U$ :
\begin{itemize}
\item $f(x) \in \H(x)$
\item il existe un voisinage compact $V$ de $x$ dans $U$ et une suite d'éléments de $\K(V)$ qui converge vers $f|_V$ pour la semi-norme uniforme sur $V$.
\end{itemize}
est noté $\O(U)$.
\smallbreak
Pour tout $x \in U$, l'application d'évaluation $\ev_x$ s'étend à $\O(U)$ et, si $f \in \O(U)$, $\ev_x(f)$ est encore noté $f(x)$.
\end{df}

Le foncteur contravariant $V \mapsto \O(V)$ est un faisceau munissant $\Aff^n_\A$ d'une structure d'espace localement annelé. Si $x$ est un point de $\Aff^n_\A$, le corps résiduel de $\O_x$ est appelé \emph{corps résiduel en $x$} et noté $\kappa(x)$. On définit à présent les notions d'espace $\A$-analytique et de morphisme d'espaces $\A$-analytiques.

\begin{df}
Soient $m,n \in \N$ et $U \subset \Aff^m_\A$ et $V \subset \Aff^n_\A$ des ouverts. Un morphisme d'espaces localement annelés $(f,f^\sharp) : (U,\O_U) \rightarrow (V,\O_V)$ est un \emph{morphisme d'espaces $\A$-analytiques} si, pour tout $x \in U$, le morphisme $f^\sharp_x$ induit un plongement isométrique de corps $\kappa(f(x)) \rightarrow \kappa(x)$.
\end{df}

\begin{df}
Soit $U \subset \Aff^n_\A$ un ouvert. Un \emph{fermé analytique de $U$} est un espace localement annelé de la forme $\left( \Supp{\O_U/\I}, \iota^{-1}(\O_U/\I) \right)$ où $\I \subset \O_U$ est un faisceau cohérent d'idéaux et $\iota : \Supp{\O_U/\I} \hookrightarrow U$ est l'inclusion canonique. Dans ce cas, on dit que $\iota$ est une \emph{immersion fermée}.
\smallbreak
Un fermé analytique d'un ouvert de $\Aff^n_\A$ est appelé \emph{modèle local d'espace $\A$-analytique}. Un morphisme entre modèles locaux d'espace $\A$-analytique est un \emph{morphisme d'espaces $\A$-analytiques} s'il provient localement d'un morphisme d'espaces $\A$-analytiques entre ouverts d'espaces affines.
\end{df}

\begin{rem}
Une immersion fermée est un morphisme d'espaces $\A$-analytiques.
\end{rem}

\begin{df}
Un \emph{espace $\A$-analytique} est un espace localement annelé localement isomorphe à un modèle local d'espace $\A$-analytique. Les notions de \emph{morphisme d'espaces $\A$-analytiques} et d'\emph{immersion fermée} se prolongent naturellement aux morphismes entre espaces $\A$-analytiques.
\end{df}

Si $x$ est un point d'un espace $\A$-analytique, le corps résiduel de l'anneau local $\O_x$ sera encore noté $\kappa(x)$.

\begin{rem}
Contrairement à la construction des espaces analytiques sur un corps ultramétrique complet présentée dans \cite{berkovich_spectral_1990}, les disques fermés ne sont pas des espaces $\A$-analytiques en général.
\end{rem}

\begin{df}
Un morphisme d'espaces $\A$-analytiques $X \rightarrow S$ est une \emph{immersion ouverte} s'il induit un isomorphisme entre $X$ et un ouvert de $S$.
\smallbreak
Un morphisme d'espaces $\A$-analytiques est une \emph{immersion} s'il s'écrit comme la composée à gauche d'une immersion ouverte par une immersion fermée.
\end{df}

\begin{prop}[{\cite[Proposition~2.3.7]{lemanissier_espaces_2020}}]\label{factorisation par fermé prop}
Soient $f : X \rightarrow S$ un morphisme d'espaces $\A$-analytiques et $\iota : Y \rightarrow S$ une immersion fermée induisant un isomorphisme $Y \cong \Supp{\O_S/\I}$ où $\I \subset \O_S$ est un faisceau cohérent d'idéaux. Alors $f$ se factorise par $\iota$ si et seulement si $f^*\I = 0$. Dans ce cas, cette factorisation est unique.
\end{prop}

\begin{prop}[{\cite[Corollaire~5.3]{poineau_espaces_2013}}]\label{kappa hensélien prop}
Soient $X$ un espace $\A$-analytique et $x \in X$. Alors le corps $\kappa(x)$ est hensélien.
\end{prop}

\begin{cor}\label{kappa ou H fini séparable cor}
Soient $f : X \rightarrow S$ un morphisme d'espaces $\A$-analytiques, $x \in X$ et $s = f(x)$. Alors les assertions suivantes sont équivalentes :
\begin{itemize}
\item[i)] $\kappa(x)$ est une extension finie séparable de $\kappa(s)$
\item[ii)] $\H(x)$ est une extension finie séparable de $\H(s)$
\end{itemize}
\end{cor}

\begin{proof}
D'après la proposition \ref{kappa hensélien prop}, $\kappa(s)$ est hensélien. Or, $\H(x)$ est le complété de $\kappa(x)$ et $\H(s)$ est celui de $\kappa(s)$. On obtient donc le résultat d'après \cite[Proposition~2.4.1]{berkovich_etale_1993}.
\end{proof}

\section{Anneaux de base géométriques}

On rappelle la définition, issue de \cite{lemanissier_espaces_2020}, d'une classe d'anneaux contenant $\Z$ et possédant de bonnes propriétés permettant d'approfondir l'étude des espaces analytiques sur ceux-ci. Parmi les résultats principaux, on notera la cohérence du faisceau structural (théorème \ref{lemme d'Oka thm}), l'existence de produits fibrés finis et d'un foncteur d'analytification des schémas (théorème \ref{existence produits fibrés analytification thm}) ainsi que des analogues du théorème de l'application finie (théorème \ref{thm application finie}) et du Nullstellensatz de Rückert (théorème \ref{nullstellensatz thm}).
\smallbreak
Dans cette section, $\A$ désigne un anneau de Banach.

\begin{df}
Soient $n \in \N$ et $V \subset \Aff^n_\A$ une partie compacte. On note $\B(V)$ le séparé complété de $\K(V)$ pour la semi-norme uniforme sur $V$. Le morphisme naturel $\A[T_1, \dots, T_n] \rightarrow \B(V)$ induit un morphisme d'espaces localement annelés $f_V : \M(\B(V)) \rightarrow \Aff^n_\A$. On dit que $V$ est \emph{spectralement convexe} si $f_V$ induit un homéomorphisme $\M(\B(V)) \rightarrow V$ ainsi qu'un isomorphisme d'espaces localement annelés $f_V^{-1}(\mathring{V}) \rightarrow \mathring{V}$.
\end{df}

On notera qu'il est démontré dans \cite{poineau_droite_2010} que tout point de $\Aff^n_\A$ admet un système fondamental de voisinages compacts et spectralement convexes.

\begin{df}
Soient $X$ un espace topologique et $x \in X$. Un système fondamental $\V_x$ de voisinages de $x$ est \emph{fin} s'il contient un système fondamental de voisinages de chacun de ses éléments.
\end{df}

\begin{df}
Soient $m,n \in \N$, $x \in \Aff^m_\A$ et $\V_x$ un système fondamental fin de voisinages compacts spectralement convexes de $x$. On dit que $\O_x$ est \emph{fortement régulier de dimension $n$ relativement à $\V_x$} si $\O_x$ est noethérien de dimension $n$ et s'il existe des éléments $f_1, \dots, f_n \in \m_x$ vérifiant :
\begin{itemize}
\item pour tous $V \in \V_x$ et $i \in [\![1, \dots, n]\!]$, $f_i$ appartient à l'image du morphisme naturel $\B(V) \rightarrow \O_x$,
\item pour tout voisinage compact $U$ de $x$, on dispose d'une famille de réels strictement positifs $(C_V)_{V \in \V_x}$ telle que, pour tout $f \in \m_x$ appartenant à l'image du morphisme naturel $\B(U) \rightarrow \O_x$ et tout élément $V \in \V_x$ contenu dans $\mathring{U}$, on dispose de $a_1, \dots, a_n \in \B(V)$ vérifiant, pour tout $i \in [\![1, \dots, n]\!]$, $\|a_i\|_V \leq C_V \|f\|_U$ et $f = a_1f_1 + \dots + a_nf_n$.
\end{itemize}
\end{df}

\begin{df}
Soient $X$ un espace localement annelé et $x \in X$. Alors $X$ \emph{satisfait le principe du prolongement analytique en $x$} si, pour tout $f \in \O_x$ non nul, on dispose d'un voisinage ouvert $U \subset X$ de $x$ tel que, pour tout $t \in U$, l'image de $f$ dans $\O_t$ est non nulle. On dit que $X$ \emph{satisfait le principe du prolongement analytique} si c'est le cas en tout point.
\end{df}

\begin{df}\label{def anneaux de base géométriques}
Un anneau de Banach $(\A,\|.\|_\A)$ est appelé \emph{un anneau de base géométrique} si $\M(\A)$ satisfait le principe du prolongement analytique et si tout $x \in \M(\A)$ admet un système fondamental $\V_x$ de voisinages d'intérieur connexe et vérifiant :
\begin{itemize}
\item $\O_x$ est fortement régulier de dimension $\leq 1$ relativement à $\V_x$,
\item Si $\H(x)$ est trivialement valué et de caractéristique positive alors, pour tout voisinage $V \in \V_x$, il existe un fermé fini $\Gamma \subset V$ tel que, pour tout $f \in \B(V)$, $\|f\|_\Gamma = \|f\|_V$ et tout point d'un tel voisinage $V$ correspond à une semi-norme ultramétrique.
\end{itemize}
\end{df}

\begin{ex}\label{exemples classiques}
Les exemples suivants sont des anneaux de base géométriques :
\begin{itemize}
\item les corps valués complets,
\item les anneaux d'entiers de corps de nombres $\A$ munis de la valeur absolue $\max_\sigma(|\sigma(.)|_\infty)$ où $\sigma$ parcourt l'ensemble des plongements complexes de $\Frac(\A)$ et $|.|_\infty$ désigne la valeur absolue usuelle sur $\C$,
\item les corps hybrides au sens de \cite[Exemple~1.1.15]{lemanissier_espaces_2020},
\item les anneaux de valuation discrète,
\item les anneaux de Dedekind trivialement valués.
\end{itemize}
\end{ex}

Dans la suite de cet article, $\A$ désignera un anneau de base géométrique.

\begin{thm}[{\cite[Théorème~11.9]{poineau_espaces_2013}}]\label{lemme d'Oka thm}
Soit $X$ un espace $\A$-analytique. Alors le faisceau structural $\O_X$ est cohérent.
\end{thm}

\begin{lem}\label{immersion plat donc ouvert lem}
Soient $\iota : X \hookrightarrow Y$ une immersion d'espaces $\A$-analytiques et $x \in X$. Alors $\iota$ est plat en $x$ si et seulement si c'est un isomorphisme local en $x$.
\end{lem}

\begin{proof}
Si $\iota$ est une immersion alors on dispose d'un idéal $I \subset \O_y$ vérifiant $\O_x \cong \O_y/I$. On suppose que $\iota^\sharp_x : \O_y \rightarrow \O_x$ est plat. Comme $\O_y$ est noethérien, on déduit de \cite[\href{https://stacks.math.columbia.edu/tag/05KK}{Tag 05KK}]{stacks_project_authors_stacks_2021}  que $I$ est engendré par un idempotent $e \in \O_y$. Or, $e \neq 1$ car $\O_y/I \neq 0$. L'anneau $\O_y$ étant local et donc connexe, on obtient $e = 0$. Cela signifie que $\O_x \cong \O_y$ et $\iota$ est un isomorphisme local en $x$. La réciproque est immédiate.
\end{proof}

\begin{lem}\label{fini condition surjection lem}
Soient $f : X \rightarrow Y$ un morphisme d'espaces $\A$-analytiques, $x \in X$ et $y = f(x)$. Si $f^\sharp_x : \O_y \rightarrow \O_x$ est surjectif alors il existe un voisinage ouvert $U \subset X$ de $x$ tel que $f|_U$ soit une immersion.
\end{lem}

\begin{proof}
L'idéal $I = \ker(f^\sharp_x) \subset \O_y$ est de type fini car $\O_y$ est noethérien et on choisit donc une famille génératrice finie d'éléments de $I$. Soit $V \subset Y$ un voisinage ouvert de $y$ sur lequel ces générateurs sont définis et $U = f^{-1}(V)$. Ils engendrent alors un faisceau d'idéaux $\I \subset \O_V$ vérifiant $\I_y = I$ et $\left( f^*(\I) \right)_x = 0$ et qui est cohérent d'après le théorème \ref{lemme d'Oka thm}. Quitte à rétrécir $U$, on peut supposer $f^*(\I) = 0$. Alors, d'après la proposition \ref{factorisation par fermé prop}, $f|_U : U \rightarrow V$ se factorise par l'immersion fermée $\iota : \Supp{\O_V/\I} \hookrightarrow V$. On note $g : U \rightarrow \Supp{\O_V/\I}$ le morphisme obtenu. Comme $f^\sharp_x$ est surjectif, le morphisme $g^\sharp_x : \O_{g(x)} \cong \O_y/I \rightarrow \O_x$ est un isomorphisme, $g$ est un isomorphisme local en $x$ et $f|_U : U \rightarrow V$ est une immersion fermée.
\end{proof}

\begin{prop}[{\cite[Proposition~4.1.1]{lemanissier_espaces_2020}}]\label{section globale représenté par A^n prop}
Soient $X$ un espace $\A$-analytique et $n \in \N$. Le foncteur $X \mapsto \Gamma(X,\O_X)^n$ est représenté par $\Aff^n_\A$. De plus, en notant $T_1, \dots, T_n$ les coordonnées de $\Aff^n_\A$, une transformation naturelle \centerrightarrow{\Hom(X,\Aff^n_\A)}{\Gamma(X,\O_X)^n} est donnée par $f \mapsto \left( f^\sharp(T_1), \dots, f^\sharp(T_n) \right)$.
\end{prop}

\begin{thm}[{\cite[Théorèmes 4.1.13 et 4.3.8]{lemanissier_espaces_2020}}]\label{existence produits fibrés analytification thm}
La catégorie des espaces $\A$-analytiques admet des produits fibrés finis ainsi qu'un foncteur d'analytification $\Xsch \mapsto \Xsch^\an$ depuis la catégorie des $\A$-schémas localement de présentation finie.
\end{thm}

\begin{nota}
Si $\Xsch$ est un $\A$-schéma localement de présentation finie, on note $\rho : \Xsch^\an \rightarrow \Xsch$ le morphisme canonique.
\end{nota}

\begin{lem}[{Preuve de \cite[Lemme~6.5.2]{lemanissier_espaces_2020}}]\label{analytification fibre lem}
Soient $f : \Xsch \rightarrow~\Ssch$ un morphisme entre $\A$-schémas localement de présentation finie et $s~\in~\Ssch^\an$. On a alors un isomorphisme canonique : \[ \left(\Xsch_{\rho(s)} \otimes_{\kappa(\rho(s))} \H(s)\right)^\an \cong \left(\Xsch^\an\right)_s. \]
\end{lem}

\begin{lem}\label{card fibre prod lem}
Soient $S$ un espace $\A$-analytique, $X$ et $Y$ des espaces au-dessus de $S$, $s \in S$ et $x \in X$ un point au-dessus de $s$ tel que $\H(x)$ soit une extension séparable de $\H(s)$ de degré $d$. On note $p_X$ et $p_Y$ les projections $X \times_S Y \rightarrow X$ et $X \times_S Y \rightarrow Y$. Alors, pour tout point $y \in Y$, l'ensemble $p_X^{-1}(\{x\}) \cap p_Y^{-1}(\{y\})$ est fini et de cardinal inférieur à $d$.
\end{lem}

\begin{proof}
Soit $y \in Y$. On commence par remarquer que si $y$ ne s'envoie pas sur $s$ alors $p_X^{-1}(\{x\}) \cap p_Y^{-1}(\{y\}) = \emptyset$. On peut donc supposer que $y \in Y_s$ et que les produits fibrés se font sur $\H(s)$. On pose $\{x\} \times_S \{y\} = p_X^{-1}(\{x\}) \cap p_Y^{-1}(\{y\})$.
\smallbreak
Soient $n, m \in \mathbb{N}$, $U$ (respectivement $V$) un voisinage ouvert de $x$ (respectivement $y$) et $i : U \hookrightarrow \mathbb{A}^n_{\H(s)}$ et $j : V \hookrightarrow \mathbb{A}^m_{\H(s)}$ des immersions. Comme $i$ et $j$ induisent une immersion $X \times_{\H(s)} Y \hookrightarrow \mathbb{A}^{n+m}_{\H(s)}$ puis une injection $\{x\} \times_{\H(s)} \{y\} \hookrightarrow \{i(x)\} \times_{\H(s)} \{j(y)\}$, on se ramène au cas $X = \Aff^n_{\H(s)}$ et $Y = \Aff^m_{\H(s)}$.
\smallbreak
On sait que $\{x\}$ (respectivement  $\{y\}$) est une partie compacte spectralement convexe de $\mathbb{A}^n_{\H(s)}$ (respectivement $\mathbb{A}^m_{\H(s)}$). Donc, d'après la proposition \cite[Proposition~4.4.9]{lemanissier_espaces_2020}, on a $\mathcal{B}(\{x\} \times_{\H(s)} \{y\}) \cong \mathcal{B}(\{x\}) \widehat{\otimes}^\mathrm{sp}_{\H(s)} \mathcal{B}(\{y\}) \cong \mathcal{H}(x) \widehat{\otimes}^\mathrm{sp}_{\H(s)} \mathcal{H}(y)$ qui s'écrit $\prod\limits_{\substack{i=1}}^{r} K_i$ où les $K_i$ sont des extensions de $\H(s)$ et $r \leq d$ d'après \cite[Proposition~III.2.2]{weil_basic_1995}. On en déduit que $\{x\} \times_{\H(s)} \{y\}$ est homéomorphe à $\mathcal{M}(\prod\limits_{\substack{i=1}}^{r} K_i)$ qui est composé de $r$ points.
\end{proof}

\begin{prop}[{\cite[Proposition~4.5.7]{lemanissier_espaces_2020}}]\label{graphe immersion prop}
Soient $S$ un espace $\A$-analytiques et $f : X \rightarrow Y$ un morphisme d'espaces $\A$-analytiques au-dessus de $S$. Alors le graphe $\Gamma_f : X \rightarrow X \times_S Y$ est une immersion.
\end{prop}

\begin{df}
Un morphisme d'espaces $\A$-analytiques est \emph{fini} s'il est fini au sens topologique, c'est-à-dire s'il est fermé à fibres finies.
\end{df}

\begin{prop}[{\cite[Proposition~5.1.8]{lemanissier_espaces_2020}}]\label{morphisme fini anneaux locaux prop}
Soient $f : X \rightarrow S$ un morphisme fini d'espaces $\A$-analytiques, $\F$ un faisceau de $\O_X$-modules et $s \in S$. Alors le morphisme naturel \centerrightarrow{(f_*\F)_s}{\prod\limits_{x \in X_s} \F_x} est un isomorphisme de $\O_s$-modules.
\end{prop}

\begin{thm}[{\cite[Théorème~5.2.6]{lemanissier_espaces_2020}}]\label{thm application finie}
Soient $f : X \rightarrow S$ un morphisme fini d'espaces $\A$-analytiques et $\F$ un faisceau cohérent sur $X$. Alors $f_* \F$ est un faisceau cohérent sur $S$.
\end{thm}

\begin{cor}\label{anneaux locaux type fini cor}
Soient $f : X \rightarrow S$ un morphisme d'espaces $\A$-analytiques, $x \in X$ et $s = f(x)$. Si $f$ est fini en $x$ alors $\O_x$ est un $\O_s$-module de présentation finie.
\end{cor}

\begin{proof}
On applique le théorème \ref{thm application finie} au faisceau structural $\O_X$ et on en déduit que $\O_x$ est de type fini sur $\O_s$. De plus, $\O_s$ étant noethérien, on en déduit que $\O_x$ est de présentation finie.
\end{proof}

\begin{thm}[{Nullstellensatz de Rückert, \cite[Théorème~5.5.5]{lemanissier_espaces_2020}}]\label{nullstellensatz thm}
Soient $X$ un espace $\A$-analytique, $\F$ un faisceau cohérent sur $X$ et $f \in \O(X)$. On suppose que $f(x) = 0$ pour tout $x \in \Supp{\F}$. Alors, pour tout $x \in X$, il existe $n \in \N$ tel que $f^n \F_x = 0$.
\end{thm}

Même si la plupart des résultats de cet article traitent des espaces analytiques sur un anneau de base géométrique quelconque, certaines propriétés d'analytification ne sont établies que sur une classe plus restreinte d'anneaux que l'on définit maintenant.

\begin{df}\label{def anneau de Dedekind analytique}
Un anneau de base géométrique $\A$ est un \emph{anneau de Dedekind analytique} si :
\begin{itemize}
\item $\A$ est un anneau de Dedekind,
\item le morphisme $\rho : \M(\A) \rightarrow \Spec(\A)$ est surjectif,
\item si $\xi \in \Spec(\A)$ est tel que $\O_\xi$ soit un corps alors, pour tout $x \in \rho^{-1}(\xi)$, $\O_x$ est un corps,
\item si $\xi \in \Spec(\A)$ est tel que $\O_\xi$ soit un anneau de valuation discrète alors, pour tout $x \in \rho^{-1}(\xi)$, $\O_x$ est un anneau de valuation discrète et $\rho^\sharp_x : \O_\xi \rightarrow \O_x$ est le morphisme de complétion.
\end{itemize}
\end{df}

\begin{ex}
Les exemples \ref{exemples classiques} sont tous des anneaux de Dedekind analytiques.
\end{ex}

\begin{thm}[{\cite[Théorème~6.6.4]{lemanissier_espaces_2020}}]\label{analytification plat Dedekind analytique thm}
On suppose que $\A$ est un anneau de Dedekind analytique. Soit $\Xsch$ un $\A$-schéma localement de présentation finie. Alors le morphisme d'analytification $\rho : \Xsch^\an \rightarrow \Xsch$ est plat.
\end{thm}

\begin{cor}\label{plat désanalytification Dedekind analytique cor}
On suppose que $\A$ est un anneau de Dedekind analytique. Soient $f : \Xsch \rightarrow \Ssch$ un morphisme entre $\A$-schémas localement de présentation finie et $x \in \Xsch^\an$. Si $f^\an$ est plat en $x$ alors $f$ est plat en $\rho(x)$.
\end{cor}

\begin{proof}
On note $\xi = \rho(x)$, $s = f^\an(x)$ et $\sigma = \rho(s) = f(\xi)$. D'après le théorème \ref{analytification plat Dedekind analytique thm}, les morphismes $\O_\xi \rightarrow \O_x$ et $\O_\sigma \rightarrow \O_s$ sont plats. En particulier, $\O_x$ est fidèlement plat sur $\O_s$ et $\O_\xi$ et on conclut par \cite[\href{https://stacks.math.columbia.edu/tag/039V}{Tag 039V}]{stacks_project_authors_stacks_2021}.
\end{proof}

\section{Morphismes rigides épais}\label{section rig épais}

Soit $\A$ un anneau de base géométrique.
\smallbreak
Dans cette section, on définit les morphismes d'espaces $\A$-analytiques rigides épais et on démontre certaines de leurs propriétés, notamment la relation entre $\m_x$ et $\m_{X_s,x}$ énoncée en introduction. Cette étude est motivée par le fait suivant : si un morphisme $f$ est non ramifié en un point $x$ au sens de la définition \ref{def non ramifié} alors il est rigide épais en $x$.

\begin{df}
Soient $f : X \rightarrow S$ un morphisme d'espaces $\A$-analytiques et $x \in X$. On dit que $f$ est \emph{rigide épais en $x$} si $\kappa(x)$ est une extension finie de $\kappa(f(x))$.
\end{df}

\begin{rem}
On écrira \og $x$ est rigide épais au-dessus de $f(x)$ \fg{} ou simplement \og $x$ est rigide épais \fg{} sans préciser le morphisme lorsque le contexte le permettra. 
\end{rem}

\begin{df}
Si $S$ est un espace $\A$-analytique, on appellera \emph{espace $S$-analytique} tout espace $\A$-analytique muni d'un morphisme vers $S$, appelé \emph{projection sur $S$}. Si $X$ et $Y$ sont deux espaces $S$-analytiques, un \emph{morphisme d'espaces $S$-analytiques} $X \rightarrow Y$ est un morphisme d'espaces $\A$-analytiques qui commute aux projections vers $S$. 
\smallbreak
Si $n$ est un entier, on appellera \emph{espace affine $S$-analytique de dimension $n$} et on notera $\Aff^n_S$ le produit fibré $\Aff^n_\A \times_\A S$ et $\pi_S : \Aff^n_S \rightarrow S$ la projection sur $S$. On appellera alors \emph{coordonnées de $\Aff^n_S$} les relevées des coordonnées de $\Aff^n_\A$ par la projection $\Aff^n_S \rightarrow \Aff^n_\A$. Si $U$ est un sous-espace $\A$-analytique de $\Aff^n_S$ muni de la projection vers $S$ induite par $\pi_S$, on dira que c'est un \emph{modèle local d'espace $S$-analytique}.
\end{df}

\begin{nota}
Si $\Ssch$ est un schéma et $n \in \N$, on notera encore $\Aff^n_\Ssch$ l'espace affine schématique de dimension $n$ au-dessus de $\Ssch$. Afin d'alléger l'écriture, $\Aff^n_{\Spec(\A)}$ sera noté $\Aff^{n,\sch}_\A$.
\end{nota}

\begin{lem}\label{espace sur S localement modèle local lem}
Soit $f : X \rightarrow S$ un morphisme d'espaces $\A$-analytiques. Alors $X$ admet un recouvrement par des modèles locaux d'espace $S$-analytique et l'inclusion d'un tel modèle local dans $X$ est un morphisme d'espaces $S$-analytiques.
\end{lem}

\begin{proof}
Soient $U \subset X$ un modèle local d'espace $\A$-analytique, $n \in \N$ et $\iota : U \hookrightarrow \Aff^n_\A$ une immersion. D'après la proposition \ref{graphe immersion prop}, le graphe $\Gamma_f : U \hookrightarrow U \times_\A S$ de $f|_U$ est une immersion. En la composant avec $\iota \times \mathrm{Id}_{S} : U \times_\A S \hookrightarrow \Aff^n_S$, on obtient un diagramme commutatif :
\begin{center}
\begin{tikzcd}
U \arrow[rr, "\Gamma", hookrightarrow] \arrow[dr, "f"] & & \Aff^n_S \arrow[dl, "\pi_S"] \\
& S &
\end{tikzcd}
\end{center}
où $\Gamma$ désigne $(\iota \times \mathrm{Id}_S) \circ \Gamma_{f} : U \hookrightarrow \Aff^n_S$. Comme $\Gamma$ est une immersion, on en déduit que $U$ est un modèle local d'espace $S$-analytique. On peut effectuer ce raisonnement pour un recouvrement de $X$ par des modèles locaux d'espace $\A$-analytique et on en déduit le résultat.
\end{proof}

\begin{lem}\label{point 0 anneau local lem}
Soient $S$ un espace $\A$-analytique, $s \in S$, $n \in \N$ et $0_s \in \Aff^n_S$ le point 0 de la fibre au-dessus de $s$. On note $T_1, \dots, T_n$ les coordonnées de $\Aff^n_S$. On a alors un monomorphisme naturel \centerrightarrow{\O_s[T_1, \dots, T_n]}{\O_{0_s}} induisant un isomorphisme entre les complétés $(T_1, \dots, T_n)$-adiques.
\smallbreak
En particulier, on a $\O_{0_s} \subset \O_s[\![T_1, \dots, T_n]\!]$.
\end{lem}

\begin{proof}
Une récurrence simple permet de se ramener au cas $n=1$. La propriété étant locale, on peut supposer que $S$ est un modèle local d'espace $\A$-analytique et on dispose alors de $m \in \N$ et d'un faisceau cohérent d'idéaux $\I \subset \O_{\Aff^m_\A}$ vérifiant $S \cong \Supp{\O_{\Aff^m_\A}/\I}$. On a alors $\Aff^1_S \cong \Supp{\O_{\Aff^{m+1}_\A}/(\pi_m)^*\I}$ où $\pi_m : \Aff^{m+1}_\A \rightarrow \Aff^m_\A$ désigne la projection sur les $m$ premières coordonnées. On note $s'$ l'image de $s$ dans $\Aff^m_\A$, $0_{s'}$ le point 0 de la fibre de $\Aff^{m+1}_\A$ au-dessus de $s'$ et $T$ la dernière coordonnée de $\Aff^{m+1}_\A$.
\smallbreak
On cherche à présent à décrire $\O_{0_{s'}}$. On pose $V \subset \Aff^m_\A$ un voisinage compact spectralement convexe de $s'$ et on se restreint à $\Aff^1_{\B(V)}$. Si $W \subset \M(\B(V))$ est une partie compacte et $t > 0$, on note $\B(W)\langle |T|\leq t \rangle$ l'algèbre des séries $\Sigma a_n T^n$ où $(a_n) \in \B(W)^\N$ est telle que $\sum \| a_n \|_W t^n$ converge. D'après \cite[Corollaire~2.6]{poineau_espaces_2013}, le morphisme naturel \centerrightarrow{\varinjlim\limits_{W \ni s', t>0} \B(W)\langle \mid T \mid \leq t \rangle}{\O_{0_{s'}},} où $W$ parcourt les voisinages compacts de $s'$ dans $\M(\B(V))$, est un isomorphisme. De plus, on a pour tout $n \in \N$ : \[ \quotient{\B(W) \langle \mid T \mid \leq t \rangle}{(T^n)} \cong \quotient{\B(W)[T]}{(T^n)}. \] Or, $\O_{s'}$ étant noethérien, $\I_{s'} \subset \O_{s'}$ est de type fini et l'exactitude à droite de $\varinjlim$ et de la complétion $T$-adique $A \mapsto \widehat{A}$ des anneaux noethériens permet donc d'écrire :
\begin{align*}
\reallywidehat{\O_{0_s}} & \cong \reallywidehat{\quotient{\O_{0_{s'}}}{\I_{s'}}}\cong \quotient{\reallywidehat{\O_{0_{s'}}}}{\I_{s'}}\\
&\cong \quotient{\left( \varprojlim_n \quotient{ \left( \varinjlim_{W \ni s', t>0} \B(W)\langle \mid T \mid \leq t \rangle \right)}{T^n} \right)}{\I_{s'}}\\
&\cong \quotient{\left( \varprojlim_n \varinjlim_{W \ni s', t>0} \left( \quotient{ \B(W)\langle \mid T \mid \leq t \rangle }{T^n} \right) \right)}{\I_{s'}}\\
&\cong \quotient{\left( \varprojlim_n \varinjlim_{W \ni s', t>0} \left( \quotient{ \B(W)[T] }{T^n} \right) \right)}{\I_{s'}}\\
&\cong \quotient{\left( \varprojlim_n \quotient{ \left( \varinjlim_{W \ni s', t>0} \B(W)[T] \right)}{T^n} \right)}{\I_{s'}}\\
&\cong \quotient{\left( \varprojlim_n \quotient{\O_{s'}[T]}{T^n} \right)}{\I_{s'}}\\
&\cong \quotient{\reallywidehat{\O_{s'}[T]}}{\I_{s'}} \cong \reallywidehat{\quotient{\O_{s'}[T]}{\I_{s'}}}\\
&\cong \reallywidehat{\O_s[T]} \cong \O_s[\![T]\!].\\
\end{align*}
\smallbreak
L'anneau $\O_s[\![T_1, \dots, T_n]\!]$ étant le complété $(T_1, \dots, T_n)$-adique de l'anneau noethérien $\O_{0_s}$, on en conclut bien l'inclusion $\O_{0_s} \subset \O_s[\![T_1, \dots, T_n]\!]$.
\end{proof}

\begin{lem}\label{rig épais pt 0 lem}
Soient $S$ un espace $\A$-analytique, $n \in \N$, $x \in \Aff^n_S$, $s = \pi_S(x)$ et $0_s$ le point $0$ de la fibre au-dessus de $s$. On suppose que $\Aff^n_S \rightarrow S$ est rigide épais en $x$. On dispose alors de polynômes $P_1 \in \O_{0_s}[S_1], P_2 \in \O_{0_s}[S_1, S_2], \dots, P_n \in \O_{0_s}[S_1, \dots, S_n]$ tels que, en notant $T_1, \dots, T_n$ les coordonnées de $\Aff^n_S$, on ait un isomorphisme : \[ \O_x \cong \quotient{\O_{0_s}[S_1, \dots, S_n]}{(P_1(S_1)-T_1, \dots, P_n(S_1, \dots, S_n)-T_n)}. \]
Cet isomorphisme reste vérifié dans la fibre au-dessus de $s$.
\end{lem}

\begin{proof}
On montre la propriété par récurrence sur $n$.
\smallbreak
Si $n = 0$ alors $\Aff^n_S = S$, $\pi_S = \Id_S$ et donc $x = 0_s$. On a donc bien $\O_x \cong \O_{0_s}$.
\smallbreak
On suppose maintenant la propriété vérifiée pour un certain $n \in \N$ et on la montre pour $n+1$. Soient $x \in \Aff^{n+1}_S$ et $x_1 \in \Aff^1_S$ sa projection sur la première coordonnée $T_1$. Soient $m \in \N$, $U \subset S$ un voisinage ouvert de $s$, $\iota : U \hookrightarrow \Aff^m_\A$ une immersion et $\eta = \iota \times_\A \Id_{\Aff^1_\A} : \Aff^1_U \hookrightarrow \Aff^{m+1}_\A$ l'immersion induite sur $\Aff^1_U$ de coordonnée $T_1$. On note $s' = \iota(s)$, $x_1' = \eta(x_1)$, $0_{s'} \in \Aff^{m+n+1}_\A$ le point 0 de la fibre au-dessus de $s'$ et $0_{x_1'} \in \Aff^{m+n+1}_\A$ le point 0 de la fibre au-dessus de $x_1'$. Alors $x_1'$ est rigide épais au-dessus de $s'$ et on note $P_1 \in \kappa(s')[S_1]$ son polynôme minimal. On commence par montrer l'isomorphisme :
\begin{equation}\label{dimension 1}
\O_{0_{x_1'}} \cong \quotient{\O_{0_{s'}}[S_1]}{(P_1(S_1)-T_1)}.
\end{equation}
Soit $V \subset \Aff^m_\A$ un voisinage compact spectralement convexe de $s'$ tel que les coefficients de $P_1$ appartiennent à $\B(V)$. La propriété recherchée étant locale, on peut se restreindre à la démontrer sur $\Aff^{n+1}_{\B(V)}$. D'après \cite[Corollaire~8.10]{poineau_espaces_2013}, $P_1$ induit un morphisme $\varphi : \Aff^{n+1}_{\B(V)} \rightarrow \Aff^{n+1}_{\B(V)}$ de sorte que $\varphi^{-1}(0_{s'}) = 0_{x_1'}$ et on a bien l'isomorphisme \eqref{dimension 1}. On dispose à présent d'un voisinage ouvert $W \subset V$ de $s'$ ainsi que d'un faisceau cohérent d'idéaux $\I \subset \O_W$ vérifiant $\O_s \cong \quotient{\O_{s'}}{\I_{s'}}$. De plus, on a $\O_{0_{x_1}} \cong \quotient{\O_{0_{x_1'}}}{\I_{s'}\O_{0_{x_1'}}}$ et $\O_{0_s} \cong \quotient{\O_{0_{s'}}}{\I_{s'}\O_{s'}}$. On obtient donc par un passage au quotient de \eqref{dimension 1} : \[ \O_{0_{x_1}} \cong \quotient{\O_{0_s}[S_1]}{(P_1(S_1)-T_1)}. \]
L'hypothèse de récurrence appliquée au point $x \in \Aff^{n+1}_S$ au-dessus de $x_1 \in \Aff^1_S$ assure l'existence de polynômes $P_2 \in \O_{0_{x_1}}[S_2], P_3 \in \O_{0_{x_1}}[S_2, S_3], \dots, P_{n+1} \in \O_{0_{x_1}}[S_2, \dots, S_{n+1}]$ vérifiant : \[ \O_x \cong \quotient{\O_{0_{x_1}}[S_2, \dots, S_{n+1}]}{(P_2(S_2) - T_2, \dots, P_{n+1}(S_2, \dots, S_{n+1})-T_{n+1})} \]
et on en déduit : \[ \O_x \cong \quotient{\O_{0_s}[S_1, \dots, S_{n+1}]}{(P_1(S_1)-T_1, \dots, P_{n+1}(S_1, \dots, S_{n+1})-T_{n+1})}. \]
Cela conclut la démonstration.
\end{proof}

\begin{prop}\label{rig épais ideal max fibre prop}
Soient $f : X \rightarrow S$ un morphisme d'espaces $\A$-analytiques, $x \in X$ et $s = f(x)$. On suppose que $f$ est rigide épais en $x$. On a alors :
\begin{center}
$\m_{X_s,x} = \sqrt{\m_{x}\O_{X_s,x}}$.
\end{center}
\end{prop}

On commence par démontrer le cas particulier des espaces affines sur $\A$.

\begin{lem}\label{rig épais ideal max fibre affine lem}
Soient $m,n \in \N$, $\pi : \Aff^{m+n}_\A \rightarrow \Aff^{m}_\A$ la projection sur les $m$ premières coordonnées, $y_n \in \Aff^{m+n}_\A$ et $y_0 = \pi(y_n)$. On suppose que $\pi$ est rigide épais en $y_n$. On a alors : \[ \m_{\pi^{-1}(y_0),y_n} = \sqrt{\m_{y_n}\O_{\pi^{-1}(y_0),y_n}}. \]
\end{lem}

\begin{proof}
On fixe l'entier $m$ et on procède par récurrence sur $n$.
\smallbreak
Si $n = 0$ alors $\pi = \mathrm{Id}_{\Aff^m}$, $\O_{\pi^{-1}(y_0),y_n} = \H(y_0)$ et on a donc bien $\m_{\pi^{-1}(y_0),y_n} = 0 = \sqrt{\m_{\Aff^{m}_\A,y_n}\O_{\pi^{-1}(y_0),y_n}}$.
\smallbreak
Supposons à présent la propriété vérifiée pour un certain $n \in \N$ et montrons la pour $n+1$. On note $S$ la dernière coordonnée de $\Aff^{m+n+1}_\A$ ainsi que $\pi_{+1} : \Aff^{m+n+1}_\A \rightarrow \Aff^m_\A$ la projection sur les $m$ premières coordonnées. Soient $y_{n+1} \in \Aff^{m+n+1}_\A$ un point rigide épais au-dessus de $\pi_{+1}(y_{n+1})$ et $y_n \in \Aff^{m+n}_\A$ sa projection sur les $m+n$ premières coordonnées. Soient $P \in \kappa(y_n)[S]$ le polynôme minimal de $y_{n+1}$ et $V \subset \Aff^{m+n}_\A$ un voisinage compact spectralement convexe de $y_n$ tel que les coefficients de $P$ soient définis sur $\B(V)$. La propriété recherchée étant locale, on peut se restreindre à la démontrer sur $\Aff^1_{\B(V)}$. D'après \cite[Théorème~8.8]{poineau_espaces_2013}, $P$ induit un morphisme $\varphi_P : \Aff^1_{\B(V)} \rightarrow \Aff^1_{\B(V)}$ de sorte que $\varphi_P(y_{n+1})$ soit le point $0$ au-dessus de $y_n$, noté $0_{y_n}$, et $\varphi_P$ induit un isomorphisme :
\begin{equation}\label{eq1}
\O_{y_{n+1}} \cong \quotient{\O_{0_{y_n}}[S]}{(P(S)-T)}
\end{equation}
où $T$ désigne la coordonnée de $\Aff^1_{\B(V)}$ au but de $\varphi_P$. D'après le lemme \ref{point 0 anneau local lem}, une fonction $f \in \O_{0_{y_n}}$ s'écrit $f = \sum\limits_{i = 0}^{+\infty} a_i T^i$ où les $a_i$ sont des éléments de $\O_{y_n}$. On a donc $f(0_{y_n}) = 0$ si et seulement si $a_0(y_n) = 0$ et on en déduit que :
\begin{equation}\label{eq2}
\m_{0_{y_n}} = \m_{y_n}\O_{0_{y_n}} + T\O_{0_{y_n}}.
\end{equation}
De même, on a $\m_{\pi_{+1}^{-1}(y_0),0_{y_n}} = \m_{\pi^{-1}(y_0),y_n}\O_{\pi_{+1}^{-1}(y_0),0_{y_n}} + T\O_{\pi_{+1}^{-1}(y_0),0_{y_n}}$. En utilisant l'hypothèse de récurrence, on obtient : \[ \m_{\pi_{+1}^{-1}(y_0),0_{y_n}} = \sqrt{\m_{y_n}\O_{\pi^{-1}(y_0),y_n}}\O_{\pi_{+1}^{-1}(y_0),0_{y_n}} + T\O_{\pi_{+1}^{-1}(y_0),0_{y_n}}. \]
L'idéal $T\O_{\pi_{+1}^{-1}(y_0),0_{y_n}}$ étant radiciel, on peut écrire : \[ \m_{\pi_{+1}^{-1}(y_0),0_{y_n}} \subset \sqrt{\m_{y_n}\O_{\pi_{+1}^{-1}(y_0),0_{y_n}} + T\O_{\pi_{+1}^{-1}(y_0),0_{y_n}}}. \]
Cela donne bien l'inclusion $\m_{\pi_{+1}^{-1}(y_0),0_{y_n}} \subset \sqrt{\m_{0_{y_n}}\O_{\pi_{+1}^{-1}(y_0),0_{y_n}}}$ d'après \eqref{eq2}. L'inclusion réciproque étant triviale, l'égalité :
\begin{equation}\label{cas point 0}
m_{\pi_{+1}^{-1}(y_0),0_{y_n}} = \sqrt{\m_{0_{y_n}}\O_{\pi_{+1}^{-1}(y_0),0_{y_n}}}
\end{equation} en découle.
\smallbreak
On remarque à présent que, sous l'isomorphisme \eqref{eq1}, l'idéal maximal de $\O_{y_{n+1}}$ est engendré par $\m_{0_{y_n}}$. En effet, on a :
\begin{equation*}
\begin{split}
\quotient{\left( \O_{0_{y_n}}[S]/(P(S)-T) \right)}{(\m_{0_{y_n}})} & \cong \quotient{\kappa(0_{y_n})[S]}{(P(S)-T)} \\
& \cong \quotient{\kappa(y_n)[S]}{(P(S))} \\
& \cong \kappa(y_{n+1}).
\end{split}
\end{equation*}
On en déduit que la préimage de l'idéal $\sqrt{\m_{y_{n+1}} \O_{\pi_{+1}^{-1}(y_0),y_{n+1}}}$ par la surjection $\O_{\pi_{+1}^{-1}(y_0),0_{y_n}}[S] \twoheadrightarrow \O_{\pi_{+1}^{-1}(y_0),y_{n+1}}$ est \[\m = \sqrt{\m_{0_{y_n}} \O_{\pi_{+1}^{-1}(y_0),0_{y_n}}[S] + (P(S)-T)\O_{\pi_{+1}^{-1}(y_0),0_{y_n}}[S]}.\] De plus, comme $0_{y_n}$ est rigide dans sa fibre, le corps résiduel de l'anneau $\O_{\pi_{+1}^{-1}(y_0),0_{y_n}}$ vaut $\H(0_{y_n})$ et on obtient :
\begin{equation*}
\begin{split}
\quotient{\O_{\pi_{+1}^{-1}(y_0),y_{n+1}}}{\sqrt{\m_{y_{n+1}}\O_{\pi_{+1}^{-1}(y_0),y_{n+1}}}} & \cong \quotient{\O_{\pi_{+1}^{-1}(y_0),0_{y_n}}[S]}{\m} \\
& \cong \quotient{\H(0_{y_n})[S]}{\sqrt{P(S)-T}}
\end{split}
\end{equation*}
d'après \ref{cas point 0}.
\smallbreak
En écrivant $\H(0_{y_n})$ comme le quotient de $\H(y_n)[T]$ par $(T)$, on déduit : \[ \quotient{\O_{\pi_{+1}^{-1}(y_0),y_{n+1}}}{\sqrt{\m_{y_{n+1}}\O_{\pi_{+1}^{-1}(y_0),y_{n+1}}}} \cong \quotient{\H(y_n)[S]}{\sqrt{P(S)}}. \] Or, d'après \cite[Lemme~1.6.23]{lemanissier_espaces_2020}, l'idéal $\sqrt{P(S)} \subset \H(y_n)[S]$ est engendré par $\mu_{y_{n+1}}(S)$, polynôme minimal de $y_{n+1}$ à coefficients dans $\H(y_n)$. On a donc à présent : \[ \quotient{\O_{\pi_{+1}^{-1}(y_0),y_{n+1}}}{\sqrt{\m_{y_{n+1}}\O_{\pi_{+1}^{-1}(y_0),y_{n+1}}}} \cong \quotient{\H(y_n)[S]}{(\mu_{y_{n+1}}(S))} \cong \H(y_{n+1}). \] On en déduit que $\sqrt{\m_{y_{n+1}}\O_{\pi_{+1}^{-1}(y_0),y_{n+1}}}$ est bien maximal, ce qui conclut la récurrence.
\end{proof}

\begin{proof}[Démonstration de la proposition \ref{rig épais ideal max fibre prop}]
La propriété énoncée étant locale, on peut supposer que $S$ est un modèle local d'espace $\A$-analytique et on dispose de $m \in \N$ et $\iota : S \hookrightarrow \Aff^m_\A$ une immersion. Alors la propriété énoncée est vérifiée pour $f$ si et seulement si elle l'est pour $\iota \circ f$ et on se ramène au cas où $S = \Aff^m_\A$.
\smallbreak
D'après le lemme \ref{espace sur S localement modèle local lem} et quitte à restreindre $X$, on peut supposer que $X$ est un modèle local d'espace $S$-analytique. On dispose à présent de $n \in \N$ et de $\Gamma : X \hookrightarrow \Aff^n_S \cong \Aff^{m+n}_\A$ une immersion d'espaces $S$-analytiques :
\begin{center}
\begin{tikzcd}
X \arrow[rr, "\Gamma", hookrightarrow] \arrow[dr, "f"] & & \Aff^{m+n}_\A \arrow[dl, "\pi_S"] \\
& \Aff^m_\A &
\end{tikzcd}.
\end{center}
D'après le lemme \ref{rig épais ideal max fibre affine lem}, on a $\m_{\pi_S^{-1}(s),\Gamma(x)} = \sqrt{\m_{\Aff^{m+n}_\A,\Gamma(x)}\O_{\pi_S^{-1}(s),\Gamma(x)}}$. Soient $U \subset \Aff^{m+n}_\A$ un voisinage ouvert de $\Gamma(x)$ et $\I \subset \O_U$ un faisceau cohérent d'idéaux vérifiant $X = \Supp{\O_U/\I}$. On a alors :
\begin{align*}
\m_{X_s,x} &= \quotient{\m_{\pi_S^{-1}(s),\Gamma(x)}}{\I_{\Gamma(x)}}\\
&= \quotient{\sqrt{\m_{\Aff^{m+n}_\A,\Gamma(x)}\O_{\pi_S^{-1}(s),\Gamma(x)}}}{\I_{\Gamma(x)}}\\
&= \sqrt{\quotient{\m_{\Aff^{m+n}_\A,\Gamma(x)}\O_{\pi_S^{-1}(s),\Gamma(x)}}{\I_{\Gamma(x)}}}\\
&= \sqrt{\m_{X,x} \O_{X_s,x}}.
\end{align*}
\end{proof}

\begin{rem}
Si les polynômes $P(S)$ et $\mu_{y_{n+1}}(S)$ engendraient le même idéal de $\H(y_n)[S]$, la preuve du lemme \ref{rig épais ideal max fibre affine lem} pourrait être écrite sans les radicaux. C'est le cas lorsque $\kappa(x)$ est une extension séparable de $\kappa(s)$ par \cite[Proposition 2.4.1]{berkovich_etale_1993}. Sous cette hypothèse, on a donc $\m_{X_s,x} = \m_x\O_{X_s,x}$.
\end{rem}

\begin{lem}\label{rig épais platitude quotient fibre lem}
Soient $S$ un espace $\A$-analytique, $f : X \rightarrow Y$ un morphisme d'espaces $S$-analytiques, $x \in X$, $y = f(x)$ et $s \in S$ l'image de $x$. Soit $I \subset \O_y$ un idéal. On suppose que $f$ est rigide épais en $x$. Si $\O_y/I\O_y \rightarrow \O_{Y_s,y}/I\O_{Y_s,y}$ est plat alors $\O_x/I\O_x \rightarrow \O_{X_s,x}/I\O_{X_s,x}$ est plat.
\end{lem}

\begin{proof}
Cette propriété étant locale, on peut supposer que $X$ est un modèle local d'espace $Y$-analytique et on dispose alors de $n \in \N$ et d'une immersion d'espaces $Y$-analytiques $\Gamma : X \hookrightarrow \Aff^n_Y$. Comme $\Gamma$ est une immersion, il suffit de montrer la propriété pour $\pi_Y : \Aff^n_Y \rightarrow Y$. On suppose donc $X = \Aff^n_Y$ et $f = \pi_Y$.
\smallbreak
On note $0_y$ le point 0 de la fibre de $X$ au-dessus de $y$. D'après le lemme \ref{rig épais pt 0 lem}, on dispose de polynômes $P_1 \in \O_{0_y}[S_1], P_2 \in \O_{0_y}[S_1, S_2], \dots, P_n \in \O_{0_y}[S_1, \dots, S_n]$ tels que, en notant $T_1, \dots, T_n$ les coordonnées de $X = \Aff^n_Y$, on ait un isomorphisme :
\begin{equation*}
\O_{x} \cong \quotient{\O_{0_y}[S_1, \dots, S_n]}{(P_1(S_1)-T_1, \dots, P_n(S_1, \dots, S_n)-T_n)}.
\end{equation*}
Il suffit de donc de montrer le résultat dans le cas où $x = 0_y$.
\smallbreak
D'après le critère local de platitude \cite[Exposé~IV, Proposition~5.6]{grothendieck_revetements_2003}, le morphisme \centerrightarrow{\quotient{\O_x}{I\O_x}}{\quotient{\O_{X_s,x}}{I\O_{X_s,x}}} est plat si et seulement si le morphisme induit \centerrightarrow{\quotient{\O_x}{\left( I+(T_1, \dots, T_n)^l \right)\O_x}}{\quotient{\O_{X_s,x}}{\left( I+(T_1, \dots, T_n)^l \right)\O_{X_s,x}}} est plat pour tout $l \geq 1$. Or, d'après le lemme \ref{point 0 anneau local lem}, ce dernier est égal au morphisme \centerrightarrow{\quotient{\left( \quotient{\O_y}{I\O_y} \right)[T_1,\dots,T_n]}{(T_1, \dots, T_n)^l}}{\quotient{\left( \quotient{\O_{Y_s,y}}{I\O_{Y_s,y}} \right)[T_1, \dots, T_n]}{(T_1, \dots, T_n)^l}} induit par $\O_y/I\O_y \rightarrow \O_{Y_s,y}/I\O_{Y_s,y}$ qui est plat par hypothèse. Ceci conclut la démonstration.
\end{proof}

\begin{df}
Soient $f : X \rightarrow S$ un morphisme d'espaces $\A$-analytiques, $x \in X$ et $s = f(x)$. On dit que $f$ est \emph{purement localement transcendant} en $x$ si $\m_x = \m_s\O_x$.
\end{df}

\begin{rem}
Cette définition coïncide avec \cite[Définition~9.9]{poineau_espaces_2013} dans le cas du morphisme $\Aff^n_\A \rightarrow \M(\A)$ d'après \cite[Théorème~1.6.26]{lemanissier_espaces_2020}.
\end{rem}

\begin{prop}\label{rig épais corps sur les fibres prop}
Soient $f : X \rightarrow S$ un morphisme d'espaces $\A$-analytiques, $x \in X$ et $s = f(x)$. On suppose que $f$ est rigide épais en $x$. Alors l'anneau $\O_{X_s,x}$ est un corps si et seulement si $f$ est purement localement transcendant en $x$.
\end{prop}

\begin{proof}
D'après le lemme \ref{rig épais platitude quotient fibre lem} appliqué à $Y = S$ et $I = \m_s$, le morphisme $\O_x/\m_s\O_x \rightarrow \O_{X_s, x}$ est plat et
\begin{center}
\begin{tikzcd}
\O_{X_s,x} \bigotimes\limits_{\O_x/\m_s\O_x} \m_x/\m_s \ar[r, hook] & \O_{X_s,x} \bigotimes\limits_{\O_x/\m_s\O_x} \O_x/\m_s\O_x \cong \O_{X_s,x}
\end{tikzcd}
\end{center}
est donc une injection. Son image coïncide avec $(\m_x/\m_s) \O_{X_s,x} \cong \m_x \O_{X_s,x}$. Donc, si $\O_{X_s,x}$ est un corps, $\O_{X_s,x} \bigotimes\limits_{\O_x/\m_s\O_x} \m_x/\m_s = 0$ et donc $\m_x/\m_s = 0$ par fidèle platitude. Dans ce cas, $\O_x/\m_s\O_x$ est bien un corps. Réciproquement, si $\O_x/\m_s \O_x$ est un corps, alors $\m_x \O_{X_s,x} = 0$ et $\O_{X_s,x}$ est un corps d'après \ref{rig épais ideal max fibre prop}.
\end{proof}

\begin{rem}
Le théorème \ref{plat fibre algébrique thm} montre que le sens direct de la proposition \ref{rig épais corps sur les fibres prop} est en fait vérifié pour tout morphisme.
\end{rem}

\section{Morphismes plats}

Cette section est consacrée à la démonstration ainsi qu'aux corollaires du théorème \ref{plat fibre algébrique thm}. Elle se fonde sur des résultats de la section~3.
\smallbreak
Commençons par rappeler la définition de morphisme plat entre espaces $\A$-analytiques.

\begin{df}
Un morphisme d'espaces $\A$-analytiques $f : X \rightarrow S$ est \emph{plat} en $x \in X$ si $f_x^\sharp : \O_s \rightarrow \O_x$ est plat avec $s = f(x)$.
\smallskip
Un morphisme $f : X \rightarrow S$ est \emph{plat} s'il l'est en tout point de $X$.
\end{df}

\begin{lem}\label{décomposition rig épais purement trans lem}
Soient $f : X \rightarrow S$ un morphisme d'espaces $\A$-analytiques, $x \in X$ et $s = f(x)$. On dispose de voisinages $U \subset X$ de $x$ et $V \subset S$ de $s$ et de $n \in \N$ tels que $f(U) \subset V$, $f|_U$ se factorise par $\pi_V : \Aff^n_V \rightarrow V$ et, en notant $y \in \Aff^n_V$ l'image de $x$, $\pi_V$ est purement localement transcendant en $y$ et $U \rightarrow \Aff^n_V$ est rigide épais en $x$.
\end{lem}

\begin{proof}
D'après le lemme \ref{espace sur S localement modèle local lem}, on dispose de $V \subset S$ un voisinage de $s$ qui est un modèle local d'espace $\A$-analytique et $U \subset f^{-1}(V)$ un voisinage de $x$ qui est un modèle local d'espace $V$-analytique. On dispose alors de $m, l \in \N$ et d'immersions $V \hookrightarrow \Aff^m_\A$ et $U \hookrightarrow \Aff^{m+l}_\A$. On note $s'$ (resp. $x'$) l'image de $s$ (resp. $x$) dans $\Aff^m_\A$ (resp. $\Aff^{m+l}_\A$). D'après \cite[Remarque~1.6.20]{lemanissier_espaces_2020} et quitte à permuter les $l$ dernières coordonnées de $\Aff^{m+l}_\A$, on dispose de $n \leq l$ tel que la projection $y'$ de $x'$ dans $\Aff^{m+n}_\A$ soit purement localement transcendante au-dessus de $s'$ et que $x'$ soit rigide épais au-dessus de $y'$. On note $y \in \Aff^n_S$ l'unique point d'image $y' \in \Aff^{m+n}_\A$. Alors $y$ est purement localement transcendant au-dessus de $s$ et $x$ est rigide épais au-dessus de $y$.
\end{proof}

\begin{thm}\label{plat fibre algébrique thm}
Soient $f : X \rightarrow S$ un morphisme d'espaces $\A$-analytiques, $x \in X$ et $s = f(x)$. Alors le morphisme \centerrightarrow{\quotient{\O_x}{\m_s\O_x}}{\O_{X_s,x}} est plat.
\end{thm}

\begin{proof}
D'après le lemme \ref{décomposition rig épais purement trans lem} et quitte à restreindre $X$ et $S$, on dispose de $n \in \N$ tel que $f$ se factorise par $\pi_S : \Aff^n_S \rightarrow S$ et, en notant $y \in \Aff^n_S$ l'image de $x$, $y$ est purement localement transcendant au-dessus de $s$ et $x$ est rigide épais au-dessus de $y$. On en déduit que $\O_y/\m_s\O_y$ est un corps et le morphisme $\O_y/\m_s\O_y \rightarrow \O_{Y_s,y}$ est donc plat. Comme $x$ est rigide épais au-dessus de $y$, on peut appliquer le lemme \ref{rig épais platitude quotient fibre lem} et on en conclut que le morphisme $\O_x/\m_s\O_x \rightarrow \O_{X_s,x}$ est plat.
\end{proof}

\begin{cor}[Critère de platitude par fibres]\label{plat par fibres cor}
Soient $S$ un espace $\A$-analytique, $f : X \rightarrow Y$ un morphisme d'espaces $S$-analytiques, $x \in X$, $y = f(x)$ et $s \in S$ l'image de $x$. On suppose que $X \rightarrow S$ est plat en $x$. Si $f_s : X_s \rightarrow Y_s$ est plat en $x$ alors $f : X \rightarrow Y$ est plat en $x$ et $Y \rightarrow S$ est plat en $y$.
\end{cor}

\begin{proof}
D'après le théorème \ref{plat fibre algébrique thm}, les morphismes $\O_x/\m_s \rightarrow \O_{X_s,x}$ et $\O_y/\m_s \rightarrow \O_{Y_s,y}$ sont plats. En composant ce dernier par le morphisme $\O_{Y_s,y} \rightarrow \O_{X_s,x}$, plat par hypothèse, on en déduit que $\O_{X_s,x}$ est plat sur $\O_y/\m_s$. Ces morphismes étant locaux, ils sont fidèlement plats. D'après \cite[\href{https://stacks.math.columbia.edu/tag/039V}{Tag 039V}]{stacks_project_authors_stacks_2021}, on en conclut que $\O_y/\m_s \rightarrow \O_x/\m_s$ est plat et donc, d'après \cite[\href{https://stacks.math.columbia.edu/tag/00MP}{Tag 00MP}]{stacks_project_authors_stacks_2021}, $\O_s \rightarrow \O_y$ et $\O_y \rightarrow \O_x$ sont plats. Ceci conclut la démonstration.
\end{proof}

\begin{cor}\label{plat analytification cor}
Soient $\Ssch$, $\Xsch$ et $\Ysch$ des $\A$-schémas localement de présentation finie, $f : \Xsch \rightarrow \Ysch$ un morphisme de $\A$-schémas au-dessus de $\Ssch$ et $S$ un espace $\A$-analytique au-dessus de $\Ssch^\an$. On note $\Xsch_S = \Xsch^\an \times_{\Ssch^\an} S$ et $\Ysch_S = \Ysch^\an \times_{\Ssch^\an} S$ et on suppose que $\Xsch_S \rightarrow S$ est plat en un point $\widetilde{x} \in \Xsch_S$. On pose $x \in \Xsch^\an$ et $\widetilde{y} \in \Ysch_S$ les images de $\widetilde{x}$. Si $f : \Xsch \rightarrow \Ysch$ est plat en $\rho(x)$ alors $f_S : \Xsch_S \rightarrow \Ysch_S$ est plat en $\widetilde{x}$ et $\Ysch_S \rightarrow S$ est plat en $\widetilde{y}$.
\end{cor}

\begin{rem}\label{plat remarque R}
\begin{itemize}
\item Soit $\Xsch$ un $\R$-schéma localement de présentation finie. Alors \[ \left( \Xsch \times_\R \Spec(\C) \right)^\an \cong \Xsch^\an \times_\R \M(\C) \] par propriété universelle.
\item Soient $f : X \rightarrow Y$ un morphisme d'espaces $\R$-analytiques et $x \in X$. Si $f_\C : X_\C \rightarrow Y_\C$ est plat en tout point de $p_X^{-1}(x)$, alors $f$ est plat en $x$. En effet, cela implique que le morphisme $\O_x \otimes_\R \C \cong \Pi_{z \in p_X^{-1}(x)}\O_z \rightarrow \O_y \otimes_\R \C \cong \Pi_{z \in p_X^{-1}(x)}\O_{f(z)}$ est plat et on conclut par \cite[\href{https://stacks.math.columbia.edu/tag/00HJ}{Tag 00HJ}]{stacks_project_authors_stacks_2021}.
\end{itemize}
\end{rem}

\begin{proof}
On pose $s \in S$ l'image de $\widetilde{x}$. Comme $f : \Xsch \rightarrow \Ysch$ est plat en $\rho(x)$, $f \times_\Ssch \Spec(\H(s)) : \Xsch \times_\Ssch \Spec(\H(s)) \longrightarrow \Ysch \times_\Ssch \Spec(\H(s))$ est plat en tout point au-dessus de $\rho(x)$. Si $\H(s)$ est archimédien, on peut supposer $\H(s) \cong \C$ d'après la remarque \ref{plat remarque R} et on déduit de \cite[Exposé~XII, Proposition~3.1]{grothendieck_revetements_2003} que $(f \times_\Ssch \Spec(\H(s)))^\an : (\Xsch \times_\Ssch \Spec(\H(s)))^\an \longrightarrow (\Ysch \times_\Ssch \Spec(\H(s)))^\an$ est plat en tout point au-dessus de $\rho(x)$. Si $\H(s)$ est non-archimédien, on obtient le même résultat par \cite[Proposition~3.4.6]{berkovich_spectral_1990}. Or, par propriété universelle, on a $(\Xsch \times_\Ssch \Spec(\H(s)))^\an \cong \left( \Xsch_S \right)_s$ et $(\Ysch \times_\Ssch \Spec(\H(s)))^\an \cong \left( \Ysch_S \right)_s$. On en déduit que le morphisme $\O_{\left( \Ysch_S \right)_s,\widetilde{y}} \longrightarrow \O_{\left( \Xsch_S \right)_s,\widetilde{x}}$ est plat. Par le corollaire \ref{plat par fibres cor}, on conclut que $\O_{\Ysch_S,\widetilde{y}} \longrightarrow \O_{\Xsch_S,\widetilde{x}}$ et $\O_s \rightarrow \O_{\Ysch_S,\widetilde{y}}$ sont plats.
\end{proof}

Si $\A$ est un anneau de Dedekind analytique alors \cite[Proposition~6.6.7]{lemanissier_espaces_2020} s'applique et on déduit des corollaires \ref{plat désanalytification Dedekind analytique cor} et \ref{plat analytification cor} le critère suivant :

\begin{cor}
Soient $\Xsch$ et $\Ysch$ des $\A$-schémas localement de présentation finie, $f : \Xsch \rightarrow \Ysch$ un morphisme de $\A$-schémas et $x \in \Xsch^\an$. On suppose que $\Xsch \rightarrow \Spec(\A)$ est plat en $\rho(x)$. Alors $f : \Xsch \rightarrow \Ysch$ est plat en $\rho(x)$ si et seulement si $f^\an : \Xsch^\an \rightarrow \Ysch^\an$ est plat en $x$.
\end{cor}

\begin{cor}\label{plat projection affine cor}
Soient $S$ un espace $\A$-analytique et $n \in \N$. Le morphisme de projection $\pi_S : \Aff^n_S \rightarrow S$ est plat.
\end{cor}

\begin{proof}
On commence par montrer que $\pi : \Aff^n_\A \rightarrow \M(\A)$ est plat. Soient $x \in \Aff^n_\A$ et $s = \pi(x)$. On montre que $\pi$ est plat en $x$. Si $\O_s$ est un corps alors le résultat est immédiat et on se ramène donc au cas où $\O_s$ est un anneau de valuation discrète dont on notera $\varpi$ une uniformisante. Il suffit de montrer que $\O_x$ est sans $\varpi$-torsion. Comme $\O_x$ est intègre, cela revient à montrer $\pi_x^\sharp(\varpi) \neq 0$. Soit $U \subset \Aff^n_\A$ un voisinage ouvert de $x$. Alors, d'après la proposition \ref{projection affine ouverte prop}, $\pi(U)$ est un voisinage ouvert de $s$. Si $\varpi$ s'annulait en tout point de $\pi(U)$, alors on déduirait du théorème \ref{nullstellensatz thm} que $\varpi$ est nilpotent dans $\O_s$, ce qui est faux. On dispose alors de $t \in \pi(U)$ vérifiant $\varpi(t) \neq 0$ et il existe $z \in U_t$ tel que $\varpi(z) = \varpi(t) \neq 0$. On en déduit que $\pi_x^\sharp(\varpi) \neq 0$ et que $\pi$ est plat en $x$.
\smallbreak
On revient à présent sur le cas général. Soient $x \in \Aff^n_S$ et $s = \pi_S(x)$. On montre que $\pi_S$ est plat en $x$. La propriété recherchée étant locale, on peut supposer que $S$ est un modèle local d'espace $\A$-analytique et on dispose alors de $m \in \N$ et d'un faisceau d'idéaux cohérent $\I \subset \O_{\Aff^m_\A}$ vérifiant $S \cong \Supp{\O_{\Aff^m_\A}/\I}$ et $\Aff^n_S \cong \Supp{\O_{\Aff^{n+m}_\A}/\I\O_{\Aff^{n+m}_\A}}$. En notant $s'$ (resp. $x'$) l'image de $s$ (resp. $x$) dans $\Aff^m_\A$ (resp. $\Aff^{n+m}_\A$), on obtient $\O_s \cong \O_{s'}/\I_{s'}$ et $\O_x \cong \O_{x'}/\I_{s'}\O_{x'}$ et on se ramène donc au cas où $S = \Aff^m_\A$. Comme cela a été démontré ci-dessus, le morphisme $\Aff^{m+n}_\A \rightarrow \M(\A)$ est plat et on obtient alors le résultat en appliquant le corollaire \ref{plat analytification cor} au morphisme $\Aff^{n+m,\sch}_\A \rightarrow \Aff^{m,\sch}_\A$.
\end{proof}

\section{Dimension algébrique}

Dans cette section, on établit des résultats sur la dimension de Krull des anneaux locaux en certains points. On notera que la stratégie de démonstration de la proposition \ref{point isolé anneau artinien prop} et du corollaire \ref{caractérisation dim algébrique cor} est la même que celle présentée dans \cite{grauert_coherent_1984} dans le cadre de la géométrie analytique complexe. 

\begin{df}
Soit $X$ un espace $\A$-analytique. Un point $x \in X$ sera dit \emph{défini par des équations} s'il existe une immersion $\{x\} \hookrightarrow X$, c'est-à-dire que l'on dispose d'un voisinage ouvert $U \subset X$ de $x$ et de fonctions $f_1, \dots, f_d \in \mathcal{O}_X(U)$ tels que $\mathrm{Supp}(\mathcal{O}_U/(f_1, \dots, f_d)) = \{x\}$.
\end{df}

\begin{rem}\label{idéal max isole rem}
Soient $X$ un espace $\A$-analytique, $x \in X$ un point défini par des équations et $g_1, \dots, g_l$ des fonctions définies sur un voisinage de $x$ et vérifiant $(g_1, \dots, g_l)\mathcal{O}_x = \mathfrak{m}_x$. Par définition, on dispose d'un ouvert $U \subset X$ et de fonctions $f_1, \dots, f_d \in \mathcal{O}(U)$ tels que $\mathrm{Supp}(\mathcal{O}_U/(f_1, \dots, f_d)) = \{x\}$. Alors $(f_1, \dots, f_d)\mathcal{O}_x \subset (g_1, \dots, g_l)\mathcal{O}_x$ et donc, quitte à restreindre $U$, on a $\mathrm{Supp}(\mathcal{O}_U/(f_1, \dots, f_d)) \supset \mathrm{Supp}(\mathcal{O}_U/(g_1, \dots, g_l))$. On en déduit que $\{x\}~=~\mathrm{Supp}(\mathcal{O}_U/(g_1, \dots, g_l))$.
\end{rem}

\begin{lem}\label{points rigides algébriques lem}
Soient $k$ un corps valué complet et $X$ un espace $k$-analytique. Alors tous les points rigides de $X$ sont définis par des équations.
\end{lem}

\begin{proof}
Soit $x \in X$ un point rigide. Les propriétés énoncées étant locales, on peut supposer que $X$ est un modèle local d'espace $k$-analytique. Soient $n \in \mathbb{N}$ un entier et $\iota : X \hookrightarrow \mathbb{A}^n_k$ une immersion. Alors $\iota(x)$ est rigide et, s'il est défini par des équations, alors $x$ l'est aussi. On se ramène donc au cas où $X = \mathbb{A}^n_k$. Le point $\xi = \rho(x) \in \mathbb{A}^{n,\mathrm{sch}}_k$ est associé à un idéal premier $I \subset k[T_1, \dots, T_n]$ égal au noyau du morphisme d'évaluation $k[T_1, \dots, T_n] \rightarrow \mathcal{H}(x)$. Comme $x$ est rigide, $\mathcal{H}(x)$ est entier sur $k$ et donc sur son sous-anneau $k[T_1, \dots, T_n]/I$. Donc, d'après \cite[Remarque~3.1.2]{bosch_algebraic_2013}, $k[T_1, \dots, T_n]/I$ est un corps, $I$ est maximal et $\xi$ est un point fermé.
Pour tout point $y \in \rho^{-1}(\xi)$, on a $\mathcal{H}(y) = \reallywidehat{k[T_1, \dots, T_n]/I} = \widehat{\kappa(\xi)}$. Comme $\kappa(\xi)$ est fini sur $k$, on en déduit que $\kappa(\xi)$ est complet et égal à $\H(y)$. 
On peut donc associer à $y$ un unique diagramme commutatif comme suit :
\begin{center}
\begin{tikzcd}
\mathcal{M}(\kappa(\xi)) \arrow[r, "\rho_{\kappa(\xi)}"] \arrow[dr, dotted, bend right, "y"] & \mathrm{Spec}(\kappa(\xi)) \arrow[r, "\xi"] & \mathbb{A}^{n,\mathrm{sch}}_k \\
& \mathbb{A}^n_k \arrow[ur, bend right, "\rho"]
\end{tikzcd}.
\end{center}
\smallbreak
L'unicité de la flèche en pointillés étant assurée par la propriété universelle de l'analytification, on en déduit que $y = x$ et donc $\rho^{-1}(\xi) = \{x\}$. Pour finir, on sait que $I$ vérifie $\{\xi\} = V(I) \subset \mathbb{A}^{n,\mathrm{sch}}_k$ et engendre un faisceau cohérent d'idéaux $\mathcal{I} \subset \mathcal{O}_{\mathbb{A}^n_k}$ vérifiant $\rho^{-1}(\xi) = \mathrm{Supp}(\mathcal{O}_{\mathbb{A}^n_k}/\mathcal{I})$. On en déduit que $\mathrm{Supp}(\mathcal{O}_{\mathbb{A}^n_k}/\mathcal{I}) = \{x\}$ et $x$ est donc bien défini par des équations.
\end{proof}

\begin{rem}
Il semble difficile de caractériser algébriquement les points définis par des équations lorsque l'anneau de base n'est pas un corps valué complet. On peut noter l'exemple de $\M(\C^\mathrm{hyb})$ dont aucun point n'est défini par des équations malgré leurs bonnes propriétés algébriques : $\M(\C^\mathrm{hyb})$ s'écrit comme l'analytifié du schéma $\Spec(\C^\mathrm{hyb})$ et, pour tout $x \in \M(\C^\mathrm{hyb})$, $\rho(x) \in \Spec(\C^\mathrm{hyb})$ est un point fermé vérifiant $\kappa(\rho(x)) = \H(x)$.
\end{rem}

\begin{prop}\label{point isolé anneau artinien prop}
Soient $X$ un espace $\A$-analytique, $x \in X$ un point défini par des équations, $U \subset X$ un voisinage ouvert de $x$ et $f_1, \dots, f_n \in \mathcal{O}(U)$. Alors $x$ est isolé dans $\mathrm{Supp}(\mathcal{O}_U/(f_1, \dots, f_n))$ si et seulement si l'anneau $\mathcal{O}_x/(f_1, \dots, f_n)\mathcal{O}_x$ est artinien.
\end{prop}

\begin{proof}
Quitte à restreindre $U$, on peut choisir des fonctions $g_1, \dots, g_l \in \mathcal{O}(U)$ telles que $\mathfrak{m}_x = (g_1, \dots, g_l)\mathcal{O}_x$.
\smallbreak
On suppose tout d'abord que $x$ est isolé dans $\mathrm{Supp}(\mathcal{O}_U/(f_1, \dots, f_n))$. Alors, d'après la remarque \ref{idéal max isole rem}, on peut restreindre $U$ pour avoir \[ \mathrm{Supp}(\mathcal{O}_U/(f_1, \dots, f_n)) = \{x\} = \mathrm{Supp}(\mathcal{O}_U/(g_1, \dots, g_l)). \] Donc, d'après le théorème \ref{nullstellensatz thm}, on a $\mathrm{rad}\left( (f_1, \dots, f_n)\mathcal{O}_U \right) = \mathrm{rad}\left( (g_1, \dots, g_l)\mathcal{O}_U \right)$. On en déduit que $\mathrm{rad}\left( (f_1, \dots, f_n)\mathcal{O}_x \right) = \mathfrak{m}_x$ et donc que $\mathcal{O}_x/(f_1, \dots, f_n)\mathcal{O}_x$ est artinien.
\smallbreak
Réciproquement, on suppose à présent que $\mathcal{O}_x/(f_1, \dots, f_n)\mathcal{O}_x$ est artinien. Alors $\mathrm{rad}\left( (f_1, \dots, f_n)\mathcal{O}_x \right) = \mathfrak{m}_x = (g_1, \dots, g_l)\mathcal{O}_x$ et on dispose d'un entier $s \in \mathbb{N}$ tel que $(g_1^s, \dots, g_l^s)\mathcal{O}_x \subset (f_1, \dots, f_n)\mathcal{O}_x$. Donc, quitte à restreindre $U$, on a $\mathrm{Supp}(\mathcal{O}_U/(g_1^s, \dots, g_l^s)) \supset \mathrm{Supp}(\mathcal{O}_U/(f_1, \dots, f_n))$. Or, $\mathrm{Supp}(\mathcal{O}_U/(g_1^s, \dots, g_l^s))$ est en bijection avec $\mathrm{Supp}(\mathcal{O}_U/(g_1, \dots, g_l))$ qui est simplement réduit à $\{x\}$ d'après la remarque \ref{idéal max isole rem}. On en déduit que $\mathrm{Supp}(\mathcal{O}_U/(f_1, \dots, f_n)) = \{x\}$.
\end{proof}

Le résultat suivant est un exemple d'application de la proposition \ref{point isolé anneau artinien prop} :

\begin{cor}\label{isolé implique alg isolé cor}
Soient $f : X \rightarrow S$ un morphisme d'espaces $\A$-analytiques, $x \in X$ un point isolé dans sa fibre et $s = f(x)$. Alors les anneaux $\O_{X_s,x}$ et $\O_x/\m_s\O_x$ sont artiniens.
\end{cor}

\begin{proof}
Comme $x$ est isolé dans $X_s$, il y est défini par des équations et on déduit de la proposition \ref{point isolé anneau artinien prop} que son anneau local $\O_{X_s,x}$ est artinien. Or, par le théorème \ref{plat fibre algébrique thm}, on sait que $\O_{X_s,x}$ est un $\O_x/\m_s\O_x$-module fidèlement plat. En particulier, l'application induite $\Spec(\O_{X_s,x}) \rightarrow \Spec(\O_x/\m_s\O_x)$ est surjective par \cite[Exposé~IV, Proposition~2.6]{grothendieck_revetements_2003} et on en déduit que $\O_x/\m_s\O_x$ est aussi artinien par \cite[\href{https://stacks.math.columbia.edu/tag/00KJ}{Tag 00KJ}]{stacks_project_authors_stacks_2021}.
\end{proof}

\begin{cor}\label{caractérisation dim algébrique cor}
Soient $X$ un espace $\A$-analytique et $x \in X$ un point défini par des équations. Alors $\dim(\mathcal{O}_x) \leq n$ si et seulement s'il existe un voisinage ouvert $U \subset X$ de $x$ et des fonctions $f_1, \dots, f_n \in \mathcal{O}(U)$ tels que $x$ soit isolé dans $\mathrm{Supp}(\mathcal{O}_U/(f_1, \dots, f_n))$.
\end{cor}

\begin{proof}
On pose $d = \dim(\mathcal{O}_x)$ et $n$ l'entier minimal tel qu'il existe un voisinage ouvert $U \subset X$ de $x$ et des fonctions $f_1, \dots, f_n \in \mathcal{O}(U)$ vérifiant $\mathrm{Supp}(\mathcal{O}_U/(f_1, \dots, f_n)) = \{x\}$. Quitte à rétrécir $U$, la proposition \ref{point isolé anneau artinien prop} assure l'existence de fonctions $g_1, \dots, g_d \in \mathcal{O}(U)$ vérifiant $\mathrm{Supp}(\mathcal{O}_U/(g_1, \dots, g_d)) = \{x\}$, ce qui implique que $n \leq d$. Le même lemme assure que $\mathcal{O}_x/(f_1, \dots, f_n)\mathcal{O}_x$ est un anneau artinien, ce qui implique que $d \leq n$. On en déduit le résultat.
\end{proof}

\section{Morphismes non ramifiés : critère par fibres}

Dans cette section, on étudie les morphismes d'espaces $\A$-analytiques non ramifiés et on démontre un critère de ramification par fibres, répondant au passage à une conjecture énoncée dans \cite[(2.2.9)]{lemanissier_topology_2019}. On ne développe pas de théorie des différentielles de Kähler.

\begin{df}\label{def non ramifié}
Un morphisme d'espaces $\A$-analytiques $f : X \rightarrow S$ est \emph{non ramifié} en $x \in X$ si $f_x^\sharp : \O_s \rightarrow \O_x$ est non ramifié avec $s = f(x)$, c'est-à-dire $\m_x = \m_s \O_x$ et $\kappa(x)$ est une extension finie séparable de $\kappa(s)$.
\smallskip
Un morphisme $f : X \rightarrow S$ est \emph{non ramifié} s'il l'est en tout point de $X$.
\end{df}

On rappelle que le corollaire \ref{kappa ou H fini séparable cor} indique que $\kappa(x)$ est une extension finie séparable de $\kappa(s)$ si et seulement si $\H(x)$ est une extension finie séparable de $\H(s)$.

\begin{rem}\label{non ramifié rigide épais rem}
Si $f$ est non ramifié en $x$ alors $f$ est rigide épais et purement localement transcendant en $x$.
\end{rem}

\begin{prop}\label{non ramifié opérations prop}
Soient $f : X \rightarrow Y$ et $g : Y \rightarrow Z$ des morphismes d'espaces $\A$-analytiques, $x \in X$ et $y = f(x)$. Alors :
\begin{itemize}
\item[i)] Si $f$ est une immersion alors $f$ est non ramifié.
\item[ii)] Si $f$ est non ramifié en $x$ et $g$ est non ramifié en $y$ alors $g \circ f$ est non ramifié en $x$.
\end{itemize}
\end{prop}

\begin{proof}
\begin{itemize}
\item[i)] Si $f$ est une immersion alors $\O_{X,x}$ est un quotient de $\O_{Y,y}$. On en déduit le résultat.
\item[ii)] Immédiat.
\end{itemize}
\end{proof}

\begin{prop}\label{non ramifié fibre prop}
Soient $f : X \rightarrow S$ un morphisme d'espaces $\A$-analytiques, $x \in X$ et $s = f(x)$. Alors $f$ est non ramifié en $x$ si et seulement si le morphisme $f_s : X_s \rightarrow \M(\H(s))$ est non ramifié en $x$.
\end{prop}

\begin{proof}
Le corollaire \ref{kappa ou H fini séparable cor} assure que $\kappa(x)$ est une extension finie séparable de $\kappa(s)$ si et seulement si $\H(x)$ est une extension finie séparable de $\H(s)$. En particulier, on peut supposer que $x$ est rigide épais.
\medbreak
Il reste à montrer que $\m_x = \m_s\O_x$ si et seulement si $\m_{X_s,x} = \m_{\M(\H(s)),s}\O_{X_s,x}$. Comme $\m_{\M(\H(s)),s} = 0$, cela revient à montrer que $\O_x/\m_s\O_x$ est un corps si et seulement si $\O_{X_s,x}$ en est un. C'est une conséquence de la proposition \ref{rig épais corps sur les fibres prop}.
\end{proof}

\begin{cor}\label{non ramifié isolé dans fibre cor}
Soient $f : X \rightarrow S$ un morphisme d'espaces $\A$-analytiques, $x \in X$ et $s = f(x)$. Si $f$ est non ramifié en $x$ alors $x$ est isolé dans $X_s$. De plus, $f : X \rightarrow S$ est fini en $x$.
\end{cor}

\begin{proof}
D'après la proposition \ref{non ramifié fibre prop}, on sait que $f_s : X_s \rightarrow \M(\H(s))$ est non ramifié en $x$ et on en déduit que $\mathcal{O}_{X_s,x}$ est un corps. Comme $x$ est rigide dans $X_s$, on déduit du lemme \ref{points rigides algébriques lem} qu'il est défini par des équations et du corollaire \ref{caractérisation dim algébrique cor} qu'il est isolé.
\smallbreak
On en déduit que $f$ est fini en $x$ par \cite[Théorème~5.2.8]{lemanissier_espaces_2020}.
\end{proof}

\begin{cor}\label{non ramifié anneaux locaux type fini cor}
Soient $f : X \rightarrow S$ un morphisme d'espaces $\A$-analytiques, $x \in X$ et $s = f(x)$. Si $f$ est non ramifié en $x$ alors $\O_x$ est un $\O_s$-module de présentation finie.
\end{cor}

\begin{proof}
C'est une conséquence des corollaires \ref{non ramifié isolé dans fibre cor} et \ref{anneaux locaux type fini cor}.
\end{proof}

\begin{cor}\label{non ramifié fibres explicites cor}
Soient $f : X \rightarrow S$ un morphisme d'espaces $\A$-analytiques. Alors $f$ est non ramifié si et seulement si, pour tout point $s \in S$, la fibre $X_s$ s'écrit $\coprod_i \M(K_i)$ où les $K_i$ sont des extensions finies séparables de $\H(s)$.
\end{cor}

\begin{proof}
On suppose tout d'abord que $f : X \rightarrow S$ est non ramifié. Soit $s \in S$. D'après la proposition \ref{non ramifié fibre prop}, $f_s : X_s \rightarrow \M(\H(s))$ est non ramifié et on déduit du corollaire \ref{non ramifié isolé dans fibre cor} que $X_s$ est discret et s'écrit donc comme une union disjointe de points. Il suffit à présent de traiter le cas $X_s = \{x\}$. Comme $f_s : X_s \rightarrow \M(\H(s))$ est non ramifié en $x$, on a $\m_{X_s,x} = 0$ et $\O_{X_s,x}$ est un corps. Le point $x$ étant rigide dans $X_s$, on en déduit que $\O_{X_s,x} \cong \H(x)$. On a donc bien $X_s \cong \M(\H(x))$ et $\H(x)$ est une extension finie séparable de $\H(s)$ car $f_s : X_s \rightarrow \M(\H(s))$ est non ramifié.
\medbreak
Réciproquement, on suppose que pour tout point $s \in S$, la fibre $X_s$ s'écrit $\coprod_i \M(K_i)$ où les $K_i$ sont des extensions finies séparables de $\H(s)$. Alors, pour tout $x \in  X_s$, $\O_{X_s,x}$ est l'un de ces $K_i$. On en déduit que $f_s : X_s \rightarrow \M(\H(s))$ est non ramifié puis que $f : X \rightarrow S$ est non ramifié par la proposition \ref{non ramifié fibre prop}.
\end{proof}

\begin{cor}\label{non ramifié analytification cor}
Soient $f : \Xsch \rightarrow \Ssch$ un morphisme entre $\A$-schémas localement de présentation finie et $x \in \Xsch^\an$. Si $f$ est non ramifié en $\rho(x)$ alors $f^\an$ est non ramifié en $x$.
\smallbreak
On suppose de plus que $\A$ est un anneau de Dedekind analytique. Alors $f$ est non ramifié en $\rho(x)$ si et seulement si $f^\an$ est non ramifié en $x$.
\end{cor}

\begin{proof}
On suppose tout d'abord que $f$ est non ramifié en $\rho(x)$. Soient $s = f^{\an}(x) \in \Ssch^{\an}$ et $\sigma = \rho_{\Ssch}(s) \in \Ssch$. Quitte à restreindre $\Xsch$, on a $\Xsch_\sigma \cong \coprod_i \Spec(K_i)$ où les $K_i$ sont des extensions finies séparables de $\kappa(\sigma)$ d'après \cite[Corollaire~8.4.4]{bosch_algebraic_2013}. Pour tout $i$, on note $\prod_j K_{i,j}$ le produit tensoriel $K_i \otimes_{\kappa(\sigma)} \H(s)$ où les $K_{i,j}$ sont des extensions finies séparables de $\H(s)$ d'après \cite[Proposition~III.2.2]{weil_basic_1995}. On obtient : \[ \Xsch_\sigma \otimes_{\kappa(\sigma)} \H(s) \cong \coprod_{i,j} \Spec(K_{i,j}). \]
Or, d'après le lemme \ref{analytification fibre lem}, on a $(\Xsch^{\an})_s \cong (\Xsch_\sigma \otimes_{\kappa(\sigma)} \H(s))^{\an}$ et on en déduit : \[ (\Xsch^{\an})_s \cong \coprod_{i,j} \M(K_{i,j}). \]
Ceci permet de conclure d'après le corollaire \ref{non ramifié fibres explicites cor}.
\medbreak
On suppose à présent que $\A$ est un anneau de Dedekind analytique et que $f^{\an}$ est non ramifié en $x$. Alors $\H(x)$ est une extension finie séparable de $\H(s)$. On note $\widetilde{\xi}$ l'image de $x$ dans $\Xsch_\sigma \otimes_{\kappa(\sigma)} \H(s)$. Comme $x$ est rigide dans $\left( \Xsch^\an \right)_s$, on a $\kappa(\widetilde{\xi}) \cong \H(x)$ et $\kappa(\widetilde{\xi})$ est donc une extension finie séparable de $\H(s)$. De plus, d'après le théorème \ref{analytification plat Dedekind analytique thm}, le morphisme $\O_{\widetilde{\xi}} \longrightarrow \O_{\left( \Xsch^\an \right)_s,x}$ est plat. Le morphisme naturel $\O_{\left( \Xsch^\an \right)_s,x} \otimes_{\O_{\widetilde{\xi}}} \m_{\widetilde{\xi}} \longrightarrow \O_{\left( \Xsch^\an \right)_s,x}$ est donc une injection d'image $\m_{\widetilde{\xi}}\O_{\left( \Xsch^\an \right)_s,x}$. Comme $\O_{\left( \Xsch^\an \right)_s,x}$ est un corps, on a $\m_{\widetilde{\xi}}\O_{\left( \Xsch^\an \right)_s,x} = 0$ et donc $\O_{\left( \Xsch^\an \right)_s,x} \otimes_{\O_{\widetilde{\xi}}} \m_{\widetilde{\xi}} = 0$. Par fidèle platitude, on en déduit $\m_{\widetilde{\xi}} = 0$ et le morphisme $\Xsch_\sigma \otimes_{\kappa(\sigma)} \H(s) \longrightarrow \Spec(\H(s))$ est donc non ramifié en $\widetilde{\xi}$. On en conclut que $f_\sigma$ est non ramifié en $\xi$ par \cite[Proposition~8.1.11]{bosch_algebraic_2013} et donc que $f$ est non ramifié en $\xi$.
\end{proof}

\begin{rem}
Sans supposer que $\A$ est un anneau de Dedekind analytique, on peut montrer que, si $f^\an$ est non ramifié en tout point de $(f^\an)^{-1}(f^\an(x))$, alors $f$ est non ramifié en $\rho(x)$ par le corollaire \ref{non ramifié fibres explicites cor}.
\end{rem}

\begin{cor}\label{non ramifié changement de base cor}
Soient $f : X \rightarrow S$ un morphisme d'espaces $\A$-analytiques et $S'$ un espace $S$-analytique. Si $f : X \rightarrow S$ est non ramifié alors $f_{S'} : X \times_S S' \rightarrow S'$ est non ramifié.
\end{cor}

\begin{proof}
Soient $s' \in S'$ et $s \in S$ l'image de $s'$. D'après le corollaire \ref{non ramifié fibres explicites cor}, $X_s$ s'écrit $\coprod_i \M(K_i)$ où les $K_i$ sont des extensions finies séparables de $\H(s)$. On obtient alors :
\begin{align*}
f_{S'}^{-1}(s') &\cong \left( X \times_S S' \right) \times_{S'} \M(\H(s'))\\
&\cong X \times_S \M(\H(s'))\\
&\cong \left( X \times_S \M(\H(s)) \right) \times_{\M(\H(s))} \M(\H(s'))\\
&\cong X_s \times_{\M(\H(s))} \M(\H(s'))\\
&\cong \coprod_i \M(K_i) \times_{\M(\H(s))} \M(\H(s'))\\
&\cong \coprod_i  \M(K_i \reallywidehat{\otimes}_{\H(s)} \H(s'))\\
&\cong \coprod_i  \M(K_i \otimes_{\H(s)} \H(s'))
\end{align*}
où le dernier isomorphisme vient du fait que les $K_i$ sont finis sur $\H(s)$. On déduit alors de \cite[Proposition~III.2.2]{weil_basic_1995} que $f_{S'}^{-1}(s')$ s'écrit $\coprod_{i,j} \M(K_{i,j})$ où les $K_{i,j}$ sont des extensions finies séparables de $\H(s')$. On conclut par le corollaire \ref{non ramifié fibres explicites cor}.
\end{proof}

\section{Morphismes non ramifiés : structure locale}

Cette section est consacrée à l'étude de la structure locale des morphismes non ramifiés.

\begin{df}\label{def morphismes étales}
Un morphisme d'espaces $\A$-analytiques $f : X \rightarrow S$ est \emph{étale} en $x \in X$ si $f_x^\sharp : \O_s \rightarrow \O_x$ est étale avec $s = f(x)$. C'est-à-dire $f_x^\sharp : \O_s \rightarrow \O_x$ est plat et non ramifié.
\smallbreak
Un morphisme $f : X \rightarrow S$ est \emph{étale} s'il l'est en tout point de $X$.
\end{df}

\begin{prop}\label{non ramifié structure locale prop}
Soient $f : X \rightarrow S$ un morphisme d'espaces $\A$-analytiques, $x \in X$ et $s = f(x)$. On suppose que $f$ est non ramifié en $x$. On dispose alors de voisinages ouverts $U \subset X$ de $x$ et $V \subset S$ de $s$ avec $f(U) \subset V$ et d'un polynôme unitaire $P \in \O(V)[T]$ de sorte que, en notant $Y = \Supp{\O_{\Aff^1_V}/(P(T))}$, $f : U \rightarrow V$ se factorise par la projection naturelle $p : Y \rightarrow V$ :
\begin{center}
\begin{tikzcd}
U \arrow[r, "\widetilde{f}"] \arrow[dr, "f"] & Y \arrow[d, "p"] \\
& V
\end{tikzcd}
\end{center}
et on a :
\begin{itemize}
\item $\kappa(x) \cong \quotient{\kappa(s)[T]}{(P(T))}$
\item $Y_s \cong \M(\H(x))$
\item $p : Y \rightarrow V$ est étale en $\widetilde{f}(x)$
\item $\widetilde{f} : U \rightarrow Y$ est une immersion
\end{itemize}
\end{prop}

\begin{proof}
On commence par noter que, d'après le corollaire \ref{non ramifié anneaux locaux type fini cor}, $\O_x$ est un $\O_s$-module de type fini. D'après le théorème de l’élément primitif, on dispose de $\bar{u} \in \kappa(x)$ non nul vérifiant $\kappa(x) \cong \kappa(s)[\bar{u}]$. On note $u \in \O_x$ un relevé de $\bar{u}$, $\bar{P} \in \kappa(s)[T]$ le polynôme minimal de $\bar{u}$ et $d = \deg(\bar{P})$. La famille d'éléments $1, \bar{u}, \dots, \bar{u}^{d-1}$ est génératrice de $\O_x/\m_s\O_x \cong \kappa(x)$ sur $\O_s$ et, d'après le lemme de Nakayama, $\O_x$ est engendré sur $\O_s$ par $1, u, \dots, u^{d-1}$. On dispose donc de $P \in \O_s[T]$ unitaire de degré $d$ vérifiant $P(u) = 0$ dans $\O_x$ et l'image de $P$ dans $\kappa(s)[T]$ s'identifie à $\bar{P}$. On a donc bien $\kappa(x) \cong \kappa(s)[T]/(P(T))$. Soient $V \subset S$ un voisinage ouvert de $s$ tel que les coefficients de $P$ soient définis sur $\O(V)$ et $U = f^{-1}(V)$. D'après le corollaire \ref{non ramifié isolé dans fibre cor}, $f$ est fini en $x$ et donc, quitte à restreindre $U$, on peut supposer que $U_s = \{x\}$. D'après \cite[Lemme~5.1.3]{lemanissier_espaces_2020} et quitte à rétrécir une nouvelle fois $U$ et $V$, on peut donc supposer $u \in \O(U)$ et $P(u) = 0$ dans $\O(U)$. D'après la proposition \ref{section globale représenté par A^n prop}, on dispose d'un unique morphisme $\widetilde{g} : U \rightarrow \Aff^1_\A$ vérifiant $\widetilde{g}^\sharp(T) = u$. On note $p_{\Aff^1_\A} : \Aff^1_V \rightarrow \Aff^1_\A$ et $p_V : \Aff^1_V \rightarrow V$ les projections naturelles. On considère $g : U \rightarrow \Aff^1_V$ l'unique morphisme vérifiant $p_{\Aff^1_\A} \circ g = \widetilde{g}$ et $p_V \circ g = f$, comme sur le diagramme suivant :
\begin{center}
\begin{tikzcd}
U \arrow[dr, dotted, "g"] \arrow[drr, bend left, "\widetilde{g}"] \arrow[ddr, bend right, "f"] & & \\
& \Aff^1_V \arrow[r] \arrow[d] & \Aff^1_\A \arrow[d] \\
& V \arrow[r] & \M(\A)
\end{tikzcd}.
\end{center}
On pose maintenant $Y = \Supp{\O_{\Aff^1_V}/(P(T))}$. On a $g^\sharp(P(T)) = P(u) = 0$ et on déduit donc de la proposition \ref{factorisation par fermé prop} que $g : U \rightarrow \Aff^1_V$ se factorise de façon unique par l'immersion fermée $Y \hookrightarrow \Aff^1_V$. Le morphisme $\widetilde{f} : U \rightarrow Y$ obtenu est celui recherché.
\smallbreak
On note $y = \widetilde{f}(x)$. D'après la proposition \ref{kappa hensélien prop}, $\kappa(s)$ est hensélien et on déduit donc de \cite[Proposition~2.4.1]{berkovich_etale_1993} que \[ \H(x) \cong \kappa(x) \otimes_{\kappa(s)} \H(s) \cong \quotient{\H(s)[T]}{(P(T))}. \]
On en conclut que $Y_s \cong \M(\H(x))$ et donc, d'après la proposition \ref{non ramifié fibre prop}, $p : Y \rightarrow V$ est non ramifié en $y$. Ce morphisme étant plat car $P(T)$ est unitaire, on en déduit qu'il est étale en $y$.
\smallbreak
On montre à présent que, quitte à restreindre $U$, $\widetilde{f} : U \rightarrow Y$ est une immersion. D'après le lemme \ref{fini condition surjection lem}, il suffit de montrer que $\widetilde{f}^\sharp_x : \O_y \rightarrow \O_x$ est surjectif. Comme $Y_s$ est réduit à un point, on déduit de la proposition \ref{morphisme fini anneaux locaux prop} que $\O_y \cong \O_s[T]/(P(T))$. On pose $\m \subset \O_s[u]$ l'image réciproque de $\m_x$ par l'injection $\O_s[u] \hookrightarrow \O_x$. L'inclusion $\O_s[u]/\m \hookrightarrow \kappa(x)$ définissant $\kappa(x)$ comme un $\O_s[u]/\m$-module de type fini, on déduit de \cite[Remarque~3.1.2]{bosch_algebraic_2013} que $\O_s[u]/\m$ est un corps et $\m \subset \O_s[u]$ est donc un idéal maximal. Alors l'image réciproque de $\m$ par la surjection $\O_y \rightarrow \O_s[u]$ est $\m_y$ et $\O_y \rightarrow \left( \O_s[u] \right)_\m$ est donc encore surjectif. Or, d'après \cite[Lemme~2.3.4]{fu_etale_2011}, $\O_x \cong \left( \O_s[u] \right)_\m$ et on en déduit que $\widetilde{f}^\sharp_x : \O_y \rightarrow \O_x$ est surjectif et que, quitte à restreindre $U$, $\widetilde{f} : U \rightarrow Y$ est une immersion.
\end{proof}

\begin{cor}\label{non ramifié fermé dans étale cor}
Soient $f : X \rightarrow S$ un morphisme d'espaces $\A$-analytiques et $x \in X$. Alors $f$ est non ramifié en $x$ si et seulement si on dispose d'un voisinage ouvert $U \subset X$ de $x$ tel que $f|_U$ se factorise en $U \hookrightarrow Y \rightarrow S$ où $U \hookrightarrow Y$ est une immersion fermée et $Y \rightarrow S$ est étale en l'image en $x$.
\end{cor}

\begin{proof}
On suppose que $f : X \rightarrow S$ est non ramifié. Alors, d'après la proposition \ref{non ramifié structure locale prop}, on dispose d'un voisinage ouvert $U \subset X$ de $x$, d'un morphisme $p : Y \rightarrow S$ étale en l'image de $x$ et d'une immersion $\widetilde{f} : U \hookrightarrow Y$ tels que $f = p \circ \widetilde{f}$. Quitte à restreindre $Y$, on peut supposer que $\widetilde{f}$ est une immersion fermée. La réciproque découle de la proposition \ref{non ramifié opérations prop}.
\end{proof}

\begin{lem}\label{extension finie ramification lem}
Soient $f : X \rightarrow S$ un morphisme d'espaces $\A$-analytiques, $x \in X$ et $s = f(x)$. Soient maintenant $K$ une extension finie de $\H(s)$ et $x_1, \dots, x_n \in X_K = X \times_S \M(K)$ les points au-dessus de $x$. On suppose que le morphisme structural $X_K \rightarrow \M(K)$ est non ramifié en $x_1, \dots, x_n$. Alors $f :  X \rightarrow S$ est non ramifié en $x$. 
\end{lem}

\begin{proof}
D'après la proposition \ref{morphisme fini anneaux locaux prop}, on a : \[ \O_{X_s,x} \otimes_{\H(s)} K \cong \prod_{i = 1}^n \O_{X_K, x_i}. \]
Or, d'après le corollaire \ref{non ramifié fibres explicites cor}, les anneaux locaux $\O_{X_K, x_i}$ sont des extensions finies séparables de $K$. On déduit alors de \cite[Lemme~8.4.7]{bosch_algebraic_2013} que $\O_{X_s,x}$ s'écrit comme produit d'extensions finies séparables de $\H(s)$. Un tel produit ne pouvant être un anneau local que s'il comporte un unique facteur, on obtient que $\O_{X_s,x}$ est un corps ainsi qu'une extension finie séparable de $\H(s)$ et donc que $f_s : X_s \rightarrow \M(\H(s))$ est non ramifié en $x$. La proposition \ref{non ramifié fibre prop} permet donc de conclure.
\end{proof}

\begin{prop}\label{non ramifié diagonale prop}
Soient $f : X \rightarrow S$ un morphisme d'espaces $\A$-analytiques et $x \in X$. Alors $f : X \rightarrow S$ est non ramifié en $x$ si et seulement si le morphisme diagonal $\Delta_f : X \rightarrow X \times_S X$ est un isomorphisme local en $x$.
\end{prop}

\begin{proof}
On suppose que $f$ est non ramifié en $x$. D'après la proposition \ref{non ramifié structure locale prop} et quitte à restreindre $X$, on a $f = \widetilde{f} \circ \iota$ avec $\iota : X \hookrightarrow Y$ une immersion fermée d'espaces $S$-analytiques et $\widetilde{f} : Y \rightarrow S$ un morphisme d'espaces $\A$-analytiques étale en $y = \iota(x)$ et vérifiant $Y_s \cong \M(\H(x))$. Comme $\H(x)$ est une extension finie séparable de $\H(s)$, \cite[Proposition~III.2.2]{weil_basic_1995} assure que $\H(x) \otimes_{\H(s)} \H(x) \cong \prod_i K_i$ où les $K_i$ sont des extensions finies séparables de $\H(x)$. On en déduit que $Y_s \times_{\H(s)} Y_s \cong \coprod_i \M(K_i)$. Comme $\Delta_{\widetilde{f}_s} : Y_s \rightarrow Y_s \times_{\H(s)} Y_s$ est une immersion d'après la proposition \ref{graphe immersion prop}, $\H(\Delta_{\widetilde{f}_s}(y)) = \H(y)$ et donc $\O_{\Delta_{\widetilde{f}_s}(y)} \cong \H(y) \cong \O_y$. On obtient que $\Delta_{\widetilde{f}_s} = (\Delta_{\widetilde{f}})_s$ est un isomorphisme local en $y$ et, en particulier, est plat en $y$. On déduit donc du corollaire \ref{plat par fibres cor} que $\Delta_{\widetilde{f}}$ est plat en $y$ et du lemme \ref{immersion plat donc ouvert lem} que, quitte à restreindre $Y$, $\Delta_{\widetilde{f}}$ est un isomorphisme. Le diagramme 
\begin{center}
\begin{tikzcd}
X \times_Y X \ar[r, "\Gamma_\iota \times_Y \Id_X"] \ar[d, swap, "\iota \times_Y \Id_X"] & X \times_S Y \times_Y X \ar[d, "\iota \times_S \Id_Y \times_Y \Id_X"] \\
Y \times_Y X \ar[r, "\Delta_{\widetilde{f}} \times_Y \Id_X"] & Y \times_S Y \times_Y X 
\end{tikzcd}
\end{center}
étant cartésien, on obtient que $\Gamma_\iota \times_Y \Id_X$ est un isomorphisme. De plus, $\Delta_\iota$ est un isomorphisme local en $x$ car $\iota$ coïncide avec l'immersion composée $X \hookrightarrow X \times_Y X \hookrightarrow Y \times_Y Y \cong Y$. On en conclut que $\Delta_f$, étant égal au morphisme composé
\begin{center}
\begin{tikzcd}
X \ar[r, "\Delta_\iota"] & X \times_Y X \ar[r, "\Gamma_\iota \times_Y \Id_X"] & X \times_S Y \times_Y X \ar[r, "\sim"] & X \times_S X
\end{tikzcd},
\end{center}
est un isomorphisme local en $x$.
\smallbreak
Réciproquement, on suppose que $\Delta_f$ est un isomorphisme local en $x$. D'après la proposition \ref{graphe immersion prop}, $\Delta_f$ est une immersion donc, quitte à rétrécir $X$, on peut supposer que c'est une immersion ouverte. De plus, d'après la proposition \ref{non ramifié fibre prop} et quitte à remplacer $X$ par $X_s$ et $S$ par $\M(\H(s))$, on peut supposer que $S = \M(k)$ où $k$ est un corps valué complet. On commence par montrer qu'il est possible de se ramener au cas où $X$ est réduit à un point rationnel.
\smallbreak
On suppose tout d'abord que la valeur absolue sur $k$ n'est pas triviale. Dans ce cas, les points rigides sont denses dans $X$ et donc, par connexité locale de $X$, on dispose d'un point rigide $x_\mathrm{rig}$ dans la composante connexe de $x$. Soient $x_1, \dots, x_n \in X \times_k \M(\H(x))$ les points au-dessus de $x$. Par propriété universelle du produit fibré, la composante connexe de $x_i$ contient un point au-dessus de $x_\mathrm{rig}$ qui est rationnel pour tout $i \in [\![ 1, \dots, n ]\!]$. D'après le lemme \ref{extension finie ramification lem}, on peut donc supposer $x_\mathrm{rig}$ rationnel et on note $\widetilde{x_\mathrm{rig}}$ la section de $f$ induite. On note maintenant $p_1$ et $p_2$ les projections $X \times_k X \rightrightarrows X$ et on considère le morphisme $h : X \rightarrow X \times_k X$ vérifiant $p_1 \circ h = \mathrm{Id}$ et $p_2 \circ h$ est la composition de $\widetilde{x_\mathrm{rig}}$ et du morphisme structural de $X$, comme sur le diagramme suivant :
\begin{center}
\begin{tikzcd}
X \arrow[r, "f"] \arrow[ddr, bend right, "\mathrm{Id}"] \arrow[dr, dotted, "h"] & \mathcal{M}(k) \arrow[dr, bend left, "\widetilde{x_\mathrm{rig}}"] & \\
& X \times_k X \arrow[r, "p_2"] \arrow[d, "p_1"] & X \arrow[d] \\
& X \arrow[r] & \mathcal{M}(k)
\end{tikzcd}.
\end{center}
On a $h^{-1}(\Delta_f(X)) = \{x_\mathrm{rig}\}$. Comme $\Delta_f$ est une immersion ouverte, on en déduit que $\{x_\mathrm{rig}\}$ est isolé. On en conclut que $x = x_\mathrm{rig}$ est un point rationnel isolé de $X$.
\smallbreak
On suppose à présent que la valeur absolue sur $k$ est triviale. Soit $K$ une extension complète non trivialement valuée de $k$. Alors on peut appliquer le raisonnement ci-dessus en un point $x_K \in X_K$ au-dessus de $x$ et on en déduit que $x_K$ est un point rigide isolé de $X_K$.  D'après \cite[Lemme~1.21]{ducros_variation_2007}, on a $\dim_{K,x_K}X_K = 0$. On en déduit par \cite[Proposition~1.22]{ducros_variation_2007} que $\dim_{k,x}X = 0$ et donc $x$ est un point rigide isolé dans $X$ que l'on peut de nouveau supposer rationnel par le lemme \ref{extension finie ramification lem}.
\smallbreak
On peut donc se ramener au cas où $X$ est réduit à un point rationnel $x$. D'après le lemme \ref{card fibre prod lem}, $X \times_k X$ est aussi réduit à un point et $\Delta_f$ induit donc un isomorphisme $X \cong X \times_k X$. On note $p_1, p_2 : X \times_k X \rightrightarrows X$ les projections naturelles. L'immersion fermée $\M(k) \cong X^{\red} \hookrightarrow X$ définie par $\m_x$ induit une immersion fermée
\begin{center}
\begin{tikzcd}
X \cong X \times_k \M(k) \arrow[r, hook] & X \times_k X \cong X
\end{tikzcd}
\end{center}
définie par $p_2^*\m_x$. On a donc $\O_x \cong \O_x/p_2^*\m_x$ et on en déduit $p_2^\sharp (\m_x) = p_2^*\m_x = 0$ et donc $\m_x = \Delta_f^\sharp  (p_2^\sharp (\m_x)) = 0$. Alors $X \cong  \M(k)$ et on en conclut que $f : X \rightarrow \M(k)$ est bien non ramifié.
\end{proof}

\begin{cor}\label{non ramifié lieu ouvert cor}
Soit $f : X \rightarrow S$ un morphisme d'espaces $\A$-analytiques. Alors l'ensemble $\{ x \in X \mid f \textit{ est non ramifié en }x \}$ est ouvert dans $X$.
\end{cor}

\begin{proof}
Ce résultat découle directement de la proposition \ref{non ramifié diagonale prop}.
\end{proof}

\begin{cor}\label{non ramifié graphe ouvert cor}
Soient $S$ un espace $\A$-analytique, $f : X \rightarrow Y$ un morphisme d'espaces $S$-analytiques et $x \in X$. On suppose que $Y \rightarrow S$ est non ramifié en $f(x)$. Alors le graphe $\Gamma_f : X \rightarrow X \times_S Y$ est un isomorphisme local en $x$.
\end{cor}

\begin{proof}
D'après la proposition \ref{non ramifié diagonale prop}, on dispose d'un voisinage ouvert $V \subset Y$ de $f(x)$ tel que $\Delta$ soit un isomorphisme $V \rightarrow V \times_S V$. On pose $U = f^{-1}(V)$. Le diagramme suivant :
\begin{center}
\begin{tikzcd}
U \arrow[r, "\Gamma_f"] \arrow[d, "f"] & U \times_S V \arrow[d, "f \times \Id_V"] \\
V \arrow[r, "\Delta"] & V \times_S V
\end{tikzcd}
\end{center}
étant cartésien, on en déduit que $\Gamma_f : U \rightarrow U \times_S V$ est aussi un isomorphisme.
\end{proof}

\begin{cor}\label{non ramifié descente fini et plat cor}
Soient $f : X \rightarrow Y$ et $g : Y \rightarrow Z$ des morphismes d'espaces $\A$-analytiques, $x \in X$ et $y = f(x)$. On suppose que $g \circ f$ est fini et plat en $x$ et que $g$ est non ramifié en $y$. Alors $f$ est fini et plat en $x$.
\end{cor}

\begin{proof}
D'après le corollaire \ref{non ramifié graphe ouvert cor}, le graphe $\Gamma_f : X \rightarrow X \times_Z Y$ est un isomorphisme local en $x$. Or, d'après \cite[Proposition~5.6.5]{lemanissier_espaces_2020}, $p_Y : X \times_Z Y \rightarrow Y$ est fini et plat en tout point au-dessus de $x$. Donc $f = p_Y \circ \Gamma_f$ est fini et plat en $x$.
\end{proof}

\begin{cor}\label{non ramifié section iso local cor}
Soit $f : X \rightarrow Y$ un morphisme non ramifié d'espaces $\A$-analytiques. Alors toute section $\sigma : Y \rightarrow X$ de $f$ est une immersion ouverte.
\end{cor}

\begin{proof}
On commence par remarquer que toutes les sections sont des immersions. En effet, si $\sigma$ est une section de $f$ alors c'est un morphisme au-dessus de $Y$ et il coïncide avec son graphe $\Gamma_\sigma : Y \rightarrow X \times_Y Y \cong X$. On déduit donc de la proposition \ref{graphe immersion prop} que $\sigma$ est une immersion.
\smallbreak
On suppose maintenant que $f$ est non ramifié. Comme $f \circ \sigma = \Id_Y$, on déduit du corollaire \ref{non ramifié descente fini et plat cor} que $\sigma$ est plat puis du lemme \ref{immersion plat donc ouvert lem} que c'est une immersion ouverte.
\end{proof}

\section{Morphismes étales}

On étudie ici les morphismes étales d'espaces $\A$-analytiques, comme définis en \ref{def morphismes étales}. On présente des résultats semblables à ceux des morphismes de schémas, avec quelques différences dans le cas de la proposition \ref{étale structure locale prop} et de son corollaire \ref{étale Chevalley cor}.

\begin{prop}\label{étale opérations prop}
Soient $f : X \rightarrow Y$, $g : Y \rightarrow Z$ et $h : Y' \rightarrow Y$ des morphismes d'espaces $\A$-analytiques, $x \in X$ et $y = f(x)$. Alors :
\begin{itemize}
\item[i)] Si $f$ est étale en $x$ et $g$ est étale en $y$ alors $g \circ f$ est étale en $x$.
\item[ii)] Si $f$ est étale en $x$ alors $f_{Y'} : X \times_Y Y' \rightarrow Y'$ est étale en tout point au-dessus de $x$.
\item[iii)] Si $g \circ f$ est étale en $x$ et $g$ est non ramifié en $y$ alors $f$ est étale en $x$.
\item[iv)] Si $f$ et $h$ sont étales alors tout morphisme $X \rightarrow Y'$ au-dessus de $Y$ est étale.
\end{itemize}
\end{prop}

\begin{proof}
\begin{itemize}
\item[i)] est immédiat.
\item[ii)] Soit $x' \in X \times_Y Y'$ un point au-dessus de $x$. D'après le corollaire \ref{non ramifié changement de base cor}, il suffit de vérifier que $f_{Y'} : X \times_Y Y' \rightarrow Y'$ est plat en $x'$. Le corollaire \ref{non ramifié isolé dans fibre cor} assure que, quitte à rétrécir $X$, on peut supposer $f$ fini. On peut donc conclure par \cite[Proposition~5.6.5]{lemanissier_espaces_2020}.
\item[iii)] D'après le corollaire \ref{non ramifié graphe ouvert cor} et le lemme \ref{immersion plat donc ouvert lem}, le graphe $\Gamma_f : X \rightarrow X \times_Z Y$ est étale en $x$. Or, d'après ii), $p_Y : X \times_Z Y \rightarrow Y$ est étale en tout point au-dessus de $x$. Donc $f = p_Y \circ \Gamma_f$ est étale en $x$ par i).
\item[iv)] C'est une conséquence de \cite[\href{https://stacks.math.columbia.edu/tag/00U7}{Tag 00U7}]{stacks_project_authors_stacks_2021}
\end{itemize}
\end{proof}

\begin{prop}\label{étale par fibres prop}
Soient $S$ un espace $\A$-analytique, $f : X \rightarrow Y$ un morphisme d'espaces $S$-analytiques, $x \in X$, $y = f(x)$ et $s \in S$ l'image de $x$. On suppose que $X \rightarrow S$ est plat en $x$. Si $f_s : X_s \rightarrow Y_s$ est étale en $x$ alors $f : X \rightarrow Y$ est étale en $x$.
\end{prop}

\begin{proof}
D'après le corollaire \ref{plat par fibres cor}, $f$ est plat en $x$ et il reste à montrer que ce morphisme est non ramifié en $x$. Comme $f_s$ est non ramifié en $x$ et $f_s^{-1}(y) = f^{-1}(y)$, on conclut par la proposition \ref{non ramifié fibre prop}.
\end{proof}

\begin{prop}\label{étale structure locale prop}
Soient $f : X \rightarrow S$ un morphisme d'espaces $\A$-analytiques, $x \in X$ et $s = f(x)$. On suppose que $f$ est étale en $x$. On dispose alors de voisinages ouverts $U \subset X$ de $x$ et $V \subset S$ de $s$ avec $f(U) \subset V$ et d'un polynôme unitaire $P \in \O(V)[T]$ de sorte que, en notant $Y = \Supp{\O_{\Aff^1_V}/(P(T))}$, $f : U \rightarrow V$ se factorise par la projection naturelle $p : Y \rightarrow V$ :
\begin{center}
\begin{tikzcd}
U \arrow[r, "\widetilde{f}"] \arrow[dr, "f"] & Y \arrow[d, "p"] \\
& V
\end{tikzcd}
\end{center}
et on a :
\begin{itemize}
\item $\kappa(x) \cong \quotient{\kappa(s)[T]}{(P(T))}$
\item $Y_s \cong \M(\H(x))$
\item $\widetilde{f} : U \rightarrow Y$ est un isomorphisme local en $x$
\end{itemize}
\end{prop}

\begin{proof}
C'est une conséquence des propositions \ref{non ramifié structure locale prop} et \ref{étale opérations prop} iii) et du lemme \ref{immersion plat donc ouvert lem}.
\end{proof}

\begin{rem}\label{étale exemple présentation locale rem}
La proposition \ref{étale structure locale prop} ne peut être écrite \emph{mutatis mutandis} dans le cadre schématique qu'au prix du fait que la fibre $Y_s$ est réduite à un point. Ce résultat a effectivement des conséquences pouvant étonner le lecteur accoutumé à la théorie des schémas (voir notamment le corollaire \ref{étale Chevalley cor}). On peut considérer l'exemple suivant : soient $\Ssch = \Spec(\Z)$, $\Xsch = \Spec(\Z[i])$, $f : \Xsch \rightarrow \Ssch$ le morphisme structural, $x \in \Xsch^\an$ le point extrême de la branche $(1+2i)$-adique, $s = f^\an(x)$, $\xi = \rho(x)$ et $\sigma = f(\xi)$. Le polynôme $Q(T) = T^2 + 1$ est irréductible dans $\O_\sigma[T] \cong \Z_{(5)}[T]$ et tout voisinage de $\xi$ se plonge bien dans $\Xsch \cong V(Q(T)) \subset \Aff^{1,\sch}_\Z$. Dans ce cas, $\Xsch_\sigma$ contient deux points distincts. Par contre, $Q(T)$ est scindé dans $\O_s[T] \cong \Z_5[T]$ et, en choisissant un de ses facteurs irréductibles $P(T)$, on dispose d'un voisinage de $x$ qui se plonge dans $\Ssch^\an \cong V(P(T)) \subset \Aff^1_\Z$.
\end{rem}

\begin{cor}\label{étale Chevalley cor}
Soient $f : X \rightarrow S$ un morphisme d'espaces $\A$-analytiques, $x \in X$ et $s = f(x)$. Alors $f$ est étale en $x$ si et seulement si on dispose d'un polynôme unitaire $P(T) \in \O_s[T]$ dont l'image dans $\kappa(s)[T]$ est irréductible et séparable et tel que $f$ induise un isomorphisme : \[ \quotient{\O_s[T]}{(P(T))} \cong \O_x. \]
\end{cor}

\begin{proof}
On suppose tout d'abord que $f : X \rightarrow S$ est étale en $x$. D'après la proposition \ref{étale structure locale prop} et quitte à restreindre $X$ et $S$, on dispose d'un polynôme unitaire $P(T) \in \O(S)[T]$ vérifiant $\kappa(x) \cong \kappa(s)[T]/(P(T))$ et d'un isomorphisme local $X \rightarrow Y = \Supp{\O_{\Aff^1_S}/(P(T))}$ en $x$. En particulier, $\O_x \cong \O_y$ où $y$ désigne l'image de $x$ dans $Y$. D'après ce même lemme, $Y_s$ est réduit à un point et on déduit donc de la proposition \ref{morphisme fini anneaux locaux prop} que $\O_y \cong \O_s[T]/(P(T))$. L'image de $P(T)$ dans $\kappa(s)[T]$ est bien irréductible et séparable car $f$ est non ramifié en $x$ et $\kappa(x) \cong \kappa(s)[T]/(P(T))$.
\smallbreak
Réciproquement, on suppose l'existence d'un polynôme unitaire $P(T) \in \O_s[T]$ dont l'image dans $\kappa(s)[T]$ est irréductible et séparable de sorte que $f : X \rightarrow S$ induise un isomorphisme $\quotient{\O_s[T]}{(P(T))} \cong \O_x$. Alors $f$ est plat en $x$ car $P(T)$ est unitaire. On montre que $f$ est non ramifié en $x$. On a : \[ \quotient{\kappa(s)[T]}{(P(T))} \cong \quotient{\O_x}{\m_s\O_x}. \]
L'image de $P(T)$ dans $\kappa(s)[T]$ étant irréductible, on en déduit que $\O_x/\m_s\O_x$ est un corps et donc que $\m_s\O_x = \m_x$. De plus, ce polynôme étant séparable, on en conclut que $\kappa(x)$ est une extension finie séparable de $\kappa(s)$ et $f : X \rightarrow S$ est non ramifié en $x$ et donc étale en $x$.
\end{proof}

\begin{cor}\label{étale lieu ouvert cor}
Soit $f : X \rightarrow S$ un morphisme d'espaces $\A$-analytiques. Alors l'ensemble $\{x \in X \mid f \textit{ est étale en } x \}$ est ouvert dans $X$.
\end{cor}

\begin{proof}
Soit $x \in X$. On suppose que $f$ est étale en $x$. D'après le corollaire \ref{non ramifié lieu ouvert cor}, il suffit de montrer l'existence d'un voisinage ouvert $U \subset X$ de $x$ tel que le morphisme $U \rightarrow S$ soit plat. D'après la proposition \ref{étale structure locale prop} et quitte à restreindre $X$ et $S$, on dispose d'un espace $\A$-analytique $Y$ plat sur $S$ et localement isomorphe à $X$ au voisinage de $x$. Il convient alors de prendre $U$ isomorphe à un ouvert de $Y$.
\end{proof}

\begin{cor}\label{étale ouvert cor}
Soient $f : X \rightarrow S$ un morphisme d'espaces $\A$-analytiques et $x \in X$. Si $f$ est étale en $x$ alors $f$ est ouvert en $x$.
\end{cor}

\begin{proof}
On note $s = f(x)$. D'après la proposition \ref{étale structure locale prop}, on peut supposer $X \cong \Supp{\O_{\Aff^1_S}/(P(T))}$ où $P(T) \in \O(S)[T]$ est unitaire et $X_s$ est réduit à un point. Soit $U \subset X$ un voisinage ouvert de $x$. Alors $U$ est un voisinage de $X_s$ et, d'après \cite[Lemme~5.1.3]{lemanissier_espaces_2020}, on dispose d'un voisinage ouvert $V \subset S$ de $s$ vérifiant $f^{-1}(V) \subset U$. On sait que l'image de $P(T)$ dans $\H(s)[T]$ est de degré strictement positif car $x \in \Supp{\O_{\Aff^1_S}/(P(T))}$. Or, $P(T)$ est unitaire et, en particulier, son coefficient dominant ne s'annule en aucun point. On en déduit que, pour tout $t \in S$, l'image de $P(T)$ dans $\H(t)[T]$ est de degré strictement positif et donc le morphisme $f^{-1}(V) \rightarrow V$ induit par $f$ est surjectif. Alors $V \subset f(f^{-1}(V)) \subset f(U)$ et on en conclut que $f$ est ouvert en $x$.
\end{proof}

\begin{prop}\label{étale analytification prop}
Soient $f : \Xsch \rightarrow \Ssch$ un morphisme entre $\A$-schémas localement de présentation finie et $x \in \Xsch^\an$. Si $f$ est étale en $\rho(x)$ alors $f^\an$ est étale en $x$.
\smallbreak
On suppose de plus que $\A$ est un anneau de Dedekind analytique. Alors $f$ est étale en $\rho(x)$ si et seulement si $f^\an$ est étale en $x$.
\end{prop}

\begin{proof}
On suppose tout d'abord que $f$ est étale en $\rho(x)$. D'après le corollaire \ref{non ramifié analytification cor}, $f^\an$ est non ramifié en $x$ et il suffit de montrer qu'il est plat en $x$. Quitte à restreindre $\Ssch$ et d'après \cite[Théorème~2.3.5]{fu_etale_2011}, on dispose d'un polynôme $P(T) \in \O(\Ssch)[T]$ de sorte que $f$ se factorise en :
\begin{center}
\begin{tikzcd}
\Xsch \ar[r,"\widetilde{f}"] \ar[rd,"f"] & \Ysch = \Supp{\quotient{\O_{\Aff^1_\Ssch}}{P(T)}} \ar[d,"\pi_\Ssch"]\\
& \Ssch
\end{tikzcd}
\end{center}
où $\widetilde{f}$ est un isomorphisme local en $\rho(x)$. Alors $\Ysch$ est localement de présentation finie sur $\A$ et $f^\an = \widetilde{f}^\an \circ \left(\pi_\Ssch\right)^\an$. Or, $\widetilde{f}^\an$ est un isomorphisme local en $x$ d'après \cite[Proposition~6.5.3]{lemanissier_espaces_2020} et $\left(\pi_\Ssch\right)^\an$ est plat en $\widetilde{f}^\an(x)$ car $P(T)$ est unitaire. On en conclut que $f^\an$ est plat en $x$ et donc étale en $x$.
\smallbreak
La seconde partie de l'énoncé découle des corollaires \ref{plat désanalytification Dedekind analytique cor} et \ref{non ramifié analytification cor}.
\end{proof}

\begin{prop}\label{étale invariance topologique prop}
Soient $S$ un espace $\A$-analytique et $S_0 \hookrightarrow S$ un fermé analytique ayant le même espace topologique sous-jacent que $S$. Le foncteur $X \mapsto X \times_S S_0$ définit une équivalence entre la catégorie des espaces $\A$-analytiques étales sur $S$ et la catégorie des espaces $\A$-analytiques étales sur $S_0$.
\end{prop}

\begin{proof}
Ce foncteur est bien défini d'après la proposition \ref{étale opérations prop} i). On commence par montrer qu'il est pleinement fidèle. Soient $X$ et $Y$ des espaces $\A$-analytiques étales sur $S$. On souhaite montrer que l'application naturelle $\Hom_S(X,Y) \rightarrow \Hom_{S_0}(X \times_S S_0, Y \times_S S_0)$ est une bijection. On note $X_0 = X \times_S S_0$ et $Y_0 = Y \times_S S_0$. On a un diagramme commutatif :
\begin{center}
\begin{tikzcd}
\Hom_S(X,Y) \arrow[r] \arrow[d] & \Hom_X(X,X \times_S Y) \arrow[d] \\
\Hom_{S_0}(X_0,Y_0) \arrow[r] & \Hom_{X_0}(X_0,X_0 \times_{S_0} Y_0)
\end{tikzcd}
\end{center}
où les flèches horizontales sont définies par $f \mapsto \Gamma_f$ et ont pour réciproque la composition à gauche par la projection sur $Y$ (resp. $Y_0$). Il suffit donc de montrer que la flèche de droite, correspondant au changement de base $X_0 \rightarrow X$, est bijective. On note que, d'après le corollaire \ref{non ramifié graphe ouvert cor}, tout élément de $\Hom_X(X,X \times_SY)$ est une immersion ouverte.
\smallbreak
On renomme à présent $X$ en $S$, $X \times_S Y$ en $X$, $X_0$ en $S_0$ et $X_0 \times_{S_0} Y_0$ en $X_0$ et on note $f : X \rightarrow S$ la projection sur $S$ et $f_0 : X_0 \rightarrow S_0$ la projection sur $S_0$. Soit $\sigma \in \Hom_S(S,X)$. Comme $\sigma$ est une immersion ouverte, son image $X' \subset X$ est un ouvert pour lequel $\sigma : S \rightarrow X'$ est un isomorphisme de réciproque $f_{\mid_{X'}}$. Réciproquement, si on choisit un ouvert $X'$ de $X$ sur lequel $f$ induit un isomorphisme, la réciproque de $f_{\mid_{X'}}$ induit bien une section de $f$ qui est une immersion ouverte d'après le corollaire \ref{non ramifié section iso local cor}. On en déduit que $\Hom_S(S,X)$ est en bijection avec l'ensemble des ouverts de $X$ sur lesquels $f$ induit un isomorphisme. De la même façon, $\Hom_{S_0}(S_0,X_0)$ est en bijection avec l'ensemble des ouverts de $X_0$ sur lesquels $f_0$ induit un isomorphisme. Comme $X$ et $X_0$ ont le même espace topologique sous-jacent, les ouverts de $X$ sont en bijection avec ceux de $X_0$ via le foncteur $X' \mapsto X' \times_S S_0$ et on se ramène donc à montrer que, si $X' \subset X$ est un ouvert et $X'_0 = X' \times_S S_0$, $f_{\mid_{X'}}$ est un isomorphisme si et seulement si $f_{0 \mid_{X'_0}}$ en est un. Soient $x_0 \in X'_0$, $s_0 = f_0(x_0)$, $x \in X'$ l'image de $x_0$ et $s = f(x)$. On dispose alors d'un idéal $I \subset \O_s$ vérifiant $\O_{s_0} \cong \O_s/I$ et $\O_{x_0} \cong \O_x/I$. Comme $f$ est étale et d'après le corollaire \ref{non ramifié anneaux locaux type fini cor}, $\O_x$ est de présentation finie et plat sur $\O_s$. On en déduit que $\O_x$ est libre sur $\O_s$ d'après \cite[\href{https://stacks.math.columbia.edu/tag/00NZ}{Tag 00NZ}]{stacks_project_authors_stacks_2021} et on note $n$ l'entier vérifiant $\O_x \cong (\O_s)^n$. Comme $\O_x/I\O_x \cong (\O_s/I)^n$ est non nul, on a bien : \[ f \text{ est un isomorphisme en } x \Leftrightarrow n = 1 \Leftrightarrow f_0 \text{ est un isomorphisme en } x_0. \]
On en conclut que le foncteur $X \mapsto X \times_S S_0$ est bien pleinement fidèle.
\smallbreak
On montre à présent que ce foncteur est essentiellement surjectif. Soient $f : X_0 \rightarrow S_0$ un morphisme étale d'espaces $\A$-analytiques et $x \in X_0$. D'après la proposition \ref{étale structure locale prop}, on dispose d'un ouvert $S' \subset S$, de $S_0' = S' \times_S S_0$, d'un voisinage ouvert $U_{0,x}$ de $x$ et d'un polynôme unitaire $P_0 \in \O(S_0')[T]$ de sorte que $U_{0,x}$ soit isomorphe à un ouvert $V_{0,x}$ de $\Supp{\O_{\Aff^1_{S_0'}}/(P_0(T))}$. Quitte à restreindre $S'$, on dispose de $P(T) \in \O(S')[T]$ un relevé unitaire de $P_0(T)$. Alors $\Supp{\O_{\Aff^1_{S_0'}}/(P_0(T))}$ est isomorphe à $\Supp{\O_{\Aff^1_{S'}}/(P(T))} \times_{S'} S_0'$ qui est un fermé analytique de $\Supp{\O_{\Aff^1_{S'}}/(P(T))}$. On note $V_{0,x}^\mathrm{c}$ le complémentaire de $V_{0,x}$ dans $\Supp{\O_{\Aff^1_{S_0'}}/(P_0(T))}$. Alors l'image de $V_{0,x}^\mathrm{c}$ dans $\Supp{\O_{\Aff^1_{S'}}/(P(T))}$ est fermée et on note $V_x$ son complémentaire. On obtient : \[ U_{0,x} \cong V_{0,x} \cong V_x \times_{S'} S_0' \cong V_x \times_S S_0. \]
On vérifie à présent que l'on peut recoller les $V_x$. Soient $x,y \in X_0$ et $V_{xy}$ (resp. $V_{yx}$) l'image de $U_{0,x} \cap U_{0,y}$ dans $V_x$ (resp. $V_y$). On a alors : \[ V_{xy} \times_S S_0 \cong U_{0,x} \cap U_{0,y} \cong V_{yx} \times_S S_0 \] et, le foncteur $X \mapsto X \times_S S_0$ étant pleinement fidèle, on dispose d'un isomorphisme $\varphi_{xy} : V_{xy} \rightarrow V_{yx}$. On considère à présent un troisième point $z \in X_0$. En remarquant : \[ (U_{0,y} \cap U_{0,x}) \cap (U_{0,y} \cap U_{0,z}) \cong (U_{0,x} \cap U_{0,y}) \cap (U_{0,x} \cap U_{0,z}), \]
on déduit : \[ (V_{yx} \cap V_{yz}) \times_S S_0 \cong (V_{xy} \cap V_{xz}) \times_S S_0 \]
et donc, par pleine fidélité de $X \mapsto X \times_S S_0$, on a : \[ \varphi_{xy}^{-1} (V_{yx} \cap V_{yz}) = V_{xy} \cap V_{xz}. \]
De plus, l'image du diagramme :
\begin{center}
\begin{tikzcd}
V_{xy} \cap V_{xz} \arrow[rr, "\varphi_{xy}"] \arrow[dr, "\varphi_{xz}"] & & V_{yx} \cap V_{yz} \arrow[dl, "\varphi_{yz}"] \\
& V_{zx} \cap V_{zy} &
\end{tikzcd}
\end{center}
par le foncteur pleinement fidèle $X \mapsto X \times_S S_0$ étant commutative, ce diagramme est lui-même commutatif. D'après \cite[\href{https://stacks.math.columbia.edu/tag/01JB}{Tag 01JB}]{stacks_project_authors_stacks_2021}, on dispose alors d'un espace localement annelé $X$ admettant un recouvrement par des ouverts isomorphes aux $V_x, x \in X_0$. On en déduit que $X$ est un espace $\A$-analytique étale sur $S$ et $X_0 \cong X \times_S S_0$. On en conclut que le foncteur $X \mapsto X \times_S S_0$ est bien essentiellement surjectif.
\end{proof}

\section{Morphismes lisses}

L'objectif de cette section est d'introduire la notion de morphisme lisse d'espaces $\A$-analytiques et de déduire certaines de ses propriétés des résultats des sections précédentes.

\begin{lem}\label{lisse lemme dim}
Soient $S$ un espace $\A$-analytique, $n \in \N$, $f : X \rightarrow \Aff^n_S$ un morphisme d'espaces $S$-analytiques, $s \in S$ et $x \in X$ un point rigide épais au-dessus de $s$. Si $f$ est étale en $x$ alors $n = \dim(\O_x)-\dim(\O_s)$.
\end{lem}

\begin{proof}
On pose $y = f(x)$. Alors $y$ est rigide épais au-dessus de $s$ et on déduit de \cite[Théorème~9.17]{poineau_espaces_2013} que $\dim(\O_y) = n + \dim(\O_s)$. Comme $f^\sharp_x : \O_y \rightarrow \O_x$ est étale, on a $\dim(\O_x) = \dim(\O_y) = n + \dim(\O_s)$ et on en déduit le résultat.
\end{proof}

\begin{df}\label{def morphisme lisse}
Un morphisme d'espaces $\A$-analytiques $f : X \rightarrow S$ est \emph{lisse} en $x \in X$ si on dispose de $n \in \N$ et d'un voisinage ouvert $U \subset X$ de $x$ de sorte que $f|_U$ se factorise en :
\begin{center}
\begin{tikzcd}
U \arrow[r,"\widetilde{f}"] \arrow[dr, "f"] & \Aff^n_S \arrow[d, "\pi_S"] \\
& S
\end{tikzcd}
\end{center}
où $\widetilde{f}$ est étale en $x$. L'entier $n$ vérifiant cette propriété est noté $\dim_xf$ et appelé \emph{dimension relative de $f$} en $x$.
\smallbreak
Un morphisme d'espaces $\A$-analytiques $f : X \rightarrow S$ est \emph{lisse} s'il l'est en tout point de $X$.
\end{df}

\begin{rem}
Dans le cadre de la définition \ref{def morphisme lisse}, $\dim_xf$ est unique. En effet, la propriété universelle du produit fibré implique l'existence d'un $\H(x)$-point $\widetilde{x}$ dans $U \times_S \M(\H(x))$. Alors le morphisme étale $\widetilde{f} \times_S \M(\H(x)) : U \times_S \M(\H(x)) \rightarrow \Aff^n_{\H(x)}$ est rigide épais en $\widetilde{x}$ et on déduit du lemme \ref{lisse lemme dim} : \[ \dim_xf = \dim(\O_{\widetilde{x}})-\dim(\H(x)) = \dim(\O_{\widetilde{x}}). \]
\end{rem}

\begin{prop}\label{lisse opérations prop}
Soient $f : X \rightarrow Y$, $g : Y \rightarrow Z$ et $h : Y' \rightarrow Y$ des morphismes d'espaces $\A$-analytiques, $x \in X$ et $y = f(x)$. Alors :
\begin{itemize}
\item[i)] Si $f$ est lisse en $x$ et $g$ est lisse en $y$ alors $g \circ f$ est lisse en $x$ et $\dim_x (g \circ f) = \dim_yg + \dim_xf$.
\item[ii)] Si $f$ est lisse en $x$ alors $f_{Y'} : X \times_Y Y' \rightarrow Y'$ est lisse de dimension relative $\dim_xf$ en tout point au-dessus de $x$.
\item[iii)] Si $g \circ f$ est lisse en $x$ et $g$ est non ramifié en $y$ alors $f$ est lisse en $x$ et $\dim_xf = \dim_x (g \circ f)$.
\end{itemize}
\end{prop}

\begin{proof}
\begin{itemize}
\item[i)] On dispose d'un voisinage $U \subset X$ de $x$ (resp. $V \subset Y$ de $y$), d'un entier $n$ (resp. $m$) et d'un morphisme $\widetilde{f} : U \rightarrow \Aff^n_Y$ (resp. $\widetilde{g} : V \rightarrow \Aff^m_Z$) au-dessus de $f$ (resp. $g$) étale en $x$ (resp. $y$). D'après la proposition \ref{étale opérations prop} ii), $\Aff^n_{\widetilde{g}} : \Aff^n_V \rightarrow \Aff^{m+n}_Z$ est étale en $\widetilde{f}(x)$. Quitte à restreindre $U$, on peut supposer $U \subset \widetilde{f}^{-1}(\Aff^n_V)$ et on obtient, d'après la proposition \ref{étale opérations prop} i), que $\Aff^n_{\widetilde{g}} \circ \widetilde{f} : U \rightarrow \Aff^{m+n}_Z$ est étale en $x$. On en conclut que $g \circ f = \pi_Z \circ \Aff^n_{\widetilde{g}} \circ \widetilde{f}$ est lisse en $x$.
\item[ii)] C'est une conséquence de la proposition \ref{étale opérations prop} ii).
\item[iii)] D'après le corollaire \ref{non ramifié graphe ouvert cor}, $\Gamma_f : X \rightarrow X \times_Z Y$ est un isomorphisme local en $x$ et est donc lisse en $x$. De plus, d'après ii), $p_Y : X \times_Z Y$ est lisse en $\Gamma_f(x)$. On déduit donc de i) que $f = p_Y \circ \Gamma_f$ est lisse en $x$.
\end{itemize}
\end{proof}

\begin{prop}\label{lisse analytification prop}
Soient $f : \Xsch \rightarrow \Ssch$ un morphisme entre $\A$-schémas localement de présentation finie et $x \in \Xsch^\an$. Si $f$ est lisse en $\rho(x)$ alors $f^\an$ est lisse en $x$.
\end{prop}

\begin{proof}
Cela découle de la proposition \ref{étale analytification prop} et du fait que $\left(\Aff^n_\Ssch\right)^\an = \Aff^n_{\Ssch^\an}$.
\end{proof}

\begin{prop}\label{lisse lieu ouvert prop}
Soient $f : X \rightarrow S$ un morphisme d'espaces $\A$-analytiques et $n \in \N$. Alors l'ensemble \[\{x \in X \mid \textrm{ $f$ est lisse en $x$ de dimension relative $n$} \}\] est ouvert dans $X$.
\end{prop}

\begin{proof}
Soit $x \in X$ un point en lequel $f$ est lisse de dimension relative $n$. On dispose d'un voisinage $U \subset X$ de $x$ et d'un morphisme $\widetilde{f} : U \rightarrow \Aff^n_S$ étale en $x$ et vérifiant $f = \pi_S \circ \widetilde{f}$. D'après le corollaire \ref{étale lieu ouvert cor} et quitte à restreindre $U$, on peut supposer que $\widetilde{f}$ est étale. Alors $f$ est lisse et $\dim_xf = n$ en tout point de $U$.
\end{proof}

\begin{prop}\label{lisse plat prop}
Soient $f : X \rightarrow S$ un morphisme d'espaces $\A$-analytiques et $x \in X$. Si $f$ est lisse en $x$ alors $f$ est plat en $x$.
\end{prop}

\begin{proof}
Quitte à restreindre $X$, on dispose de $n \in \N$ et d'un morphisme $\widetilde{f} : X \rightarrow \Aff^n_S$ étale en $x$ vérifiant $f = \pi_S \circ \widetilde{f}$. Alors $\widetilde{f}$ est plat en $x$ et $\pi_S$ est plat en $\widetilde{f}(x)$ d'après le corollaire \ref{plat projection affine cor}. On en conclut que $f$ est plat en $x$.
\end{proof}

\begin{cor}\label{lisse non ramifié implique étale cor}
Un morphisme d'espaces $\A$-analytiques est étale en un point si et seulement s'il est lisse et non ramifié en ce point.
\end{cor}

\begin{proof}
Le sens direct est immédiat. La réciproque découle du fait que les morphismes lisses sont plats, qui est une conséquence de la proposition \ref{lisse plat prop}.
\end{proof}

\printbibliography

\end{document}